\documentclass[12pt]{amsart}
\usepackage{fancyhdr,amsmath,amssymb,latexsym,verbatim,tikz}
\usepackage[total={6.5in,9.5in},centering,includefoot,includehead]{geometry}
\usepackage{graphicx}
\usepackage{caption}
\usepackage{subcaption}
\usetikzlibrary{arrows}

\begingroup
 
\newtheorem*{thm*}{Theorem}
\newtheorem{thm}[subsection]{Theorem} 

\newtheorem{lemma}[subsection]{Lemma}       
\newtheorem{prop}[subsection]{Proposition} 
\newtheorem{defn}[subsection]{Definition}
\newtheorem{ex}[subsection]{Example}
\newtheorem{cor}[subsection]{Corollary}

\newtheorem{remark}{Remark}  
\endgroup
 
\newcommand{\h}{{\mathfrak{H}}}
\renewcommand{\O}{{\mathcal{O}}}
\newcommand{\U}{{\mathcal{U}}}

\newcommand{\w}{{\mathfrak{w}}}
\renewcommand{\u}{{\mathfrak{u}}}
\newcommand{\Q}{{\mathbb{Q}}}
\newcommand{\R}{{\mathbb{R}}}

\newcommand{\N}{{\mathbb{N}}}

\newcommand{\Z}{{\mathbb{Z}}}
\newcommand{\n}{{\mathfrak{n}}}
\newcommand{\ngg}{\mathfrak{n}_G}
\newcommand{\zz}{{\mathfrak{z}}}

\newcommand{\om}{{\omega}}
\newcommand{\V}{{\mathfrak{v}}}

\newcommand{\GN}{{\Gamma \setminus N}}

\newcommand{\metric}{{\langle\: ,\,  \rangle }}
\newcommand{\spn}{{\mbox{span}}}

\newcommand{\metn}{\{\mathfrak{n}, \metric\}}

\newcommand{\ds}{\displaystyle}
\newcommand{\pf}{\operatorname{pf}}
\newcommand{\sgn}{\operatorname{sgn}}
\newcommand{\Span}{\operatorname{span}} 
\newcommand{\partg}{\frac{\partial g}{\partial a_i}} 
\newcommand{\mb}{\mathbf}  

\title[]{Graphs and  Metric 2-step Nilpotent Lie Algebras}

\author{Rachelle DeCoste}
\address{Mathematics Department, Wheaton College, Norton, MA 02766}
\email{decoste\_rachelle@wheatoncollege.edu}

\author{Lisa DeMeyer}
\address{Department of Mathematics, Central Michigan University, Mount Pleasant, MI 48859}
\email{demey1la@cmich.edu}

\author{Meera Mainkar}
\address{Department of Mathematics, Central Michigan University, Mount Pleasant, MI 48859}
\email{maink1m@cmich.edu}

\date{\today}
\begin{document}

\pagestyle{fancy} 
\begin{abstract}

Dani and Mainkar introduced a method for constructing a 2-step nilpotent Lie algebra $\n_G$ from a simple directed graph $G$ in 2005. There is a natural inner product on $\ngg$ arising from the construction. We study geometric properties of the associated simply connected 2-step nilpotent Lie group $N$ with Lie algebra $\ngg$. We classify singularity properties of the Lie algebra $\ngg$ in terms of the graph $G$.  A comprehensive description is given of graphs $G$ which give rise to Heisenberg-like Lie algebras.  Conditions are given on the graph $G$ and on a lattice $\Gamma \subseteq N$ for which the quotient $\GN$, a compact nilmanifold, has a dense set of smoothly closed geodesics. 


\end{abstract}

\maketitle

\section{Introduction}

In this paper we study the geometry of 2-step nilpotent Lie algebras arising from finite, simple, directed graphs. The construction of a 2-step nilpotent Lie algebra $\n_G$ from a graph $G$ was introduced by Dani-Mainkar in \cite{DM}. Dani-Mainkar used this construction to study automorphisms of the Lie algebra which give rise to Anosov automorphisms on nilmanifolds. From the construction given in \cite{DM}, there is a natural way to construct an inner product $\metric$ on the 2-step nilpotent Lie algebra, $\n_G$. It is the aim of this paper to describe geometric properties of the simply connected 2-step nilpotent Lie group $N$ with Lie algebra $\{\n_G, \metric\}$ in terms of the graph theoretic properties of $G$.  

We study the geometry of $\{\ngg, \metric\}$ in terms of the $j$ map introduced by Kaplan in \cite{K}. Let $\zz$ denote the center of a 2-step nilpotent Lie algebra $\n$ and let $\V = \zz^{\perp}$. Then for each nonzero $Z\in \zz$, there is a linear, skew-symmetric map $j(Z):\V \rightarrow \V$ defined by $\langle j(Z)X, Y\rangle = \langle [X,Y], Z\rangle$. The $j(Z)$ maps capture both the bracket structure and the metric structure of $\metn$ and are essential in studying the geometry of the associated simply connected  Lie group $N$ and its quotients. In Section \ref{singsect}, we classify the singularity properties of the $j(Z)$ maps of $\n_G$ in terms of the graph $G$. 

In 2000 \cite{GM}, Gornet-Mast introduced Heisenberg-like Lie algebras as a generalization of the well-studied Heisenberg-type Lie algebras. Heisenberg-like Lie algebras are characterized by an abundance of 3-dimensional totally geodesic subalgebras. In \cite{GM}, Gornet and Mast explicitly computed the length spectrum of Heisenberg-like nilmanifolds. In \cite{DDM} a characterization  is given for Heisenberg-like nilpotent Lie groups in terms of the Riemannian curvature tensor. In Section \ref{Heisenbergsect}, Corollary \ref{hlike}, we prove the following result.

\begin{thm*} A 2-step nilpotent Lie algebra $\n_G$ constructed from a connected graph $G$ is Heisenberg-like if and only if the graph $G$ is a star graph or a complete graph on three vertices.
\label{bigthm:StarK3}
\end{thm*}

 The remaining results of the paper concern the density of closed geodesics  property (DCG) for nilmanifolds arising from 2-step nilpotent Lie algebras associated to graphs. A simply connected 2-step nilpotent Lie group $N$ is diffeomorphic to $\R^n$. A lattice $\Gamma \subseteq N$ is a discrete cocompact subgroup of $N$. The quotient spaces $\GN$ give the simplest nonabelian examples of a compact quotient of a simply connected Lie group and are, in a sense, similar to tori. Thus the existence of a dense subset of smoothly closed geodesics in $\GN$ is a natural question and we say that if the smoothly closed geodesics in a nilmanifold $\GN$ are dense, then the nilmanifold has the density of closed geodesics property (DCG).

In Section \ref{DCGsect}, Corollary \ref{Cor:DCGPK4}, building on results of \cite{Ma} and \cite{LP} we prove the following result.

\begin{thm*} Let $G$ be a connected graph on four vertices which is not a star graph and let $\n_G$ be the associated 2-step nilpotent Lie algebra. Then for any lattice $\Gamma$ in $N$, $\GN$ has the density of closed geodesics property.
\label{bigthm:4Vertices}
\end{thm*}

In \cite{E}, Eberlein introduced the first hit map approach, outlined in subsection \ref{subsec:FirstHitApproach}, to demonstrate DCG for any lattice in 2-step nilpotent groups with one dimensional center. This approach was used in \cite{Ma}, \cite{LP}, \cite{DeM}, to show certain infinite classes of Lie groups $N$ have the DCG property for any lattice in $N$. In \cite{DC}, the first hit map approach was used to prove DCG for an infinite class of lattices in Lie groups constructed from irreducible representations of  simple compact Lie groups.  In Theorem \ref{ThreePathMaxRankProp} and Corollary \ref{ThreePathMaxRankCor}, the first hit map approach is applied to show the DCG property holds for any lattice in the simply connected Lie group whose Lie algebra is $\n_G$, where $G$ is the path on 3 vertices. These results give the following theorem.

\begin{thm*} Let $G$ be the path on three vertices and let $\n_G$ be the metric 2-step nilpotent Lie algebra associated to $G$.  Then for any lattice $\Gamma$ in $N$, $\GN$ has the density of closed geodesics property.
\label{bigthm:3path}
\end{thm*}

In subsections \ref{stargraphs} and \ref{K3graphSubsection} we present two cases where the traditional first hit map approach fails, namely for $\n_G$ arising from the graphs $G=K_{1,n}$, the star graph on $n+1$ vertices, and from $G=K_3$, the complete graph on three vertices. Note that these are precisely the graphs which give rise to Heisenberg-like $\n_G$. In these cases, a direct approach is given. Combining Theorems \ref{DCGforStar4orMore} and \ref{thm:K3DCGP}, we prove the following result.

\begin{thm*} Let $\n_G$ be a 2-step nilpotent Lie algebra constructed from a graph $G$ where $G$ is either a star graph or a complete graph on three vertices. Then there exists a lattice $\Gamma \subseteq N$ such that $\GN$ has the density of closed geodesics property. 
\label{bigthm:K3orStar}
\end{thm*}

\section{Notation and Definitions}

Let $\metn$ be a 2-step nilpotent Lie algebra with a left invariant metric and $N$  the simply-connected, 2-step nilpotent metric Lie group with Lie algebra $\n$.  Let $\zz$ denote the center of $\n$ and let $\V$ denote the orthogonal complement of $\zz$.

\begin{defn}\label{jZdefn} {\rm{ For each nonzero $Z\in\zz$ define a skew symmetric linear transformation $j(Z):\V\rightarrow \V$ by $\langle [X,Y],Z\rangle=\langle j(Z)X,Y\rangle$ for all $X,Y\in\V$ and $Z\in\zz$. }}\end{defn}

The maps $j(Z)$ encode the metric and algebraic data of $\metn$. Hence they can be used to describe the geometry of $N$.  

Let $G$ be a finite, simple, directed graph. Notation for undirected graphs is used throughout, with direction of edges specified when necessary. For graph theoretic notation and background, see \cite{B}. We use the following standard notation from graph theory.

\begin{defn}{\rm {
A {\it complete graph} on $n$ vertices, denoted $K_n$, is the graph which has an edge between every pair of distinct vertices. A {\it complete bipartite graph}, denoted $K_{m,n}$ is a graph whose vertex set is the union of two disjoint subsets $S_1$ and $S_2$ with $|S_1|=m$ and $|S_2|=n$ where every pair of vertices in different subsets is adjacent and no pair of vertices from the same subset is adjacent. A {\it star graph} is a complete bipartite graph where one partition has order 1, denoted $K_{1,n}$. }}
\end{defn}

 Assume that $G$ is a directed graph with at least one edge.   Denote the vertices of $G$ by $S=\{X_1, \ldots X_m\}$, the edges of $G$ by $E=\{Z_1, \ldots, Z_q\}$, and denote $G=G(S,E)$ when needed for clarity. From $G$, we construct a 2-step nilpotent Lie algebra $\n_G$ as follows. The Lie algebra $\n_G$ is the vector space direct sum $\n_G= \V \oplus \zz$ where we let $E$ be a basis over $\R$ for $\zz$ and $S$ be a basis over $\R$ for $\V$. Define the bracket relations among elements of $S$ according to adjacency rules: if $Z_k$ is a directed edge from vertex $X_i$ to vertex $X_l$ then define the skew-symmetric bracket $[X_i, X_l]=Z_k$. If there is no edge between two vertices, then define the bracket of those two elements in $S$ to be zero. Extend the bracket relation to all of $\V$ using bilinearity of the bracket. Choose the inner product on $\n_G$ so that $S\cup E$ is an orthonormal basis for $\n_G$. Observe that if $Z_k$ is a directed edge from $X_i$ to $X_l$ then $j(Z_k)X_i=X_l$ .  Note that $j(Z_k)X_p=0$ for any other $X_p\in S$ where $p\neq i,l$. The maps $j(Z):\V\rightarrow \V$  are then given by definition \ref{jZdefn} and are linear over $Z\in \zz$.  We use the notation $\n$ for $\n_G$ when the notation is clear.

\begin{ex} Let $G$ be the directed graph with vertices $\{X_1, X_2, X_3, X_4\}$ and edges $\{Z_1, Z_2, Z_3\}$ as in Figure \ref{graphk13}. Construct  the metric, 2-step nilpotent Lie algebra $\n_G$ from the graph $G$ as described above. Then $\n_G$ has orthonormal basis $\{X_1, X_2, X_3, X_4, Z_1, Z_2, Z_3\}$ and skew symmetric bracket relations determined by $[X_1, X_2]=Z_1$, $[X_1, X_3]= Z_2$, $[X_1, X_4]=Z_3$ with $[X_i, X_j]=0$ when vertices $X_i,X_j$ are not adjacent. 

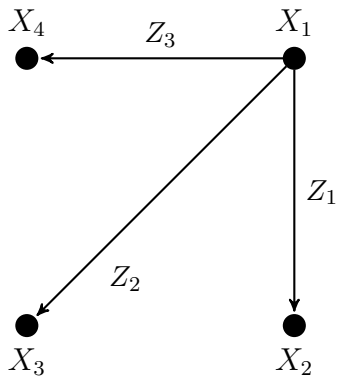
\begin{figure}[h!]
\begin{tikzpicture}[->,>=stealth',shorten >=1pt,auto,
  thick,vertex/.style={circle,draw,fill,scale=.7,font=\sffamily\Large\bfseries},node distance=2in,thick]

\node[vertex, label=above:{$X_4$}](X4) {};
  \node[vertex, label=below:{$X_3$}](X3) [below of=X4] {};
  \node[vertex, label=above:{$X_1$}] (X1) [right of=X4]  {};
  \node[vertex, label=below:{$X_2$}] (X2)[below of=X1]  {};
%

  \path[every node/.style={font=\sffamily\small}]
    (X1) edge node {$Z_1$}  (X2)
		     edge node [near end] {$Z_2$}  (X3)
				 edge node [above] {$Z_3$}  (X4);
\end{tikzpicture}
\caption{Graph $K_{1,3}$}
\label{graphk13} 
\end{figure}

\end{ex}

\begin{remark} 
\begin{enumerate}
\item If $X_i$ and $X_k$ are adjacent vertices in a directed graph $G$, the edge from vertex $X_i$ to vertex $X_k$ may be denoted $X_iX_k$. 
\item Let $G=(S,E)$ be a directed graph with an edge $Z$ from $X$ to $Y$.  Then in the metric 2-step nilpotent Lie algebra $\n_G$, we have $[X,Y]=Z$.  Let $G'=(S,E')$ be the directed graph that is identical to $G$ except that the direction of the edge from $X$ to $Y$ has been reversed, resulting in the metric Lie algebra $\n_{G'}$ where all brackets among the standard basis elements are the same as in $\n_G$ except $[X,Y]=-Z$.  Then the Lie algebras $\n_G$ and $\n_{G'}$ are isometric.
\end{enumerate}
\end{remark}

\section{Singular, Nonsingular and Almost Nonsingular Lie Algebras}  \label{singsect}

In this section we classify 2-step nilpotent Lie algebras arising from graphs as singular, nonsingular or almost nonsingular  based on the properties of the graph. These are algebraic, not metric, properties of a 2-step nilpotent Lie algebra. 

\begin{defn} \label{singnonalmost}
{\rm A  2-step nilpotent Lie algebra $\n=\V\oplus\zz$  is {\em singular } if $j(Z)$ is singular for every nonzero $Z\in\zz$.  The Lie algebra $\n$ is {\em almost nonsingular} if $j(Z)$ is nonsingular for every $Z$ in an open dense subset of $\zz$.  The Lie algebra $\n$ is {\em nonsingular} if $j(Z)$ is nonsingular for every nonzero $Z\in\zz$.}
\end{defn}

\begin{lemma}[\cite{GM} Lemma 1.16]  
Every 2-step nilpotent Lie algebra is nonsingular, almost nonsingular or singular. \label{GMsing}
\end{lemma}

Given a metric, 2-step nilpotent Lie algebra $\n$, if two nonzero elements $Z, Z' \in \zz$ can be found such that $j(Z)$ is nonsingular and $j(Z')$ is singular, then the Lie algebra $\n$ is almost nonsingular. 

\begin{lemma} \label{K2nonsinglemma} Let $G$ be a graph with at least one edge. The Lie algebra $\n_G$ is nonsingular if and only if $G=K_2$, the complete graph on two vertices.

\end{lemma}

\begin{proof}
The Lie algebra $\n_{K_2}$ is the 3-dimensional Heisenberg Lie algebra $\{X_1,X_2,Z\}$, which is nonsingular. 

Assume that $G$ has at least one edge. If $G$ is a graph which is not isomorphic to $K_2$ then $G$ contains an edge $Z$ and a vertex $X$ for which the edge $Z$ is not incident to the vertex $X$. Then $j(Z)X = 0$, hence $\n_G$ is either almost non-singular or singular. 
\end{proof}

Let $G=(S,E)$ be a graph with $|S|>1$ odd, then $\n_G$ is singular.  For each nonzero $Z\in \zz$, the map $j(Z)$ is skew-symmetric on an odd dimensional vector space $\V$. Therefore  $0$ is  an eigenvalue of $j(Z)$ for all $Z\in \zz$ and hence $\ngg$ itself is singular.

If $|S|$ is even and greater than 2 for a graph $G$, the associated 2-step nilpotent Lie algebra $\n_G$ will be either  singular or almost nonsingular by  Lemmas \ref{GMsing} and \ref{K2nonsinglemma}.  Let $E=\{Z_1,\dots,Z_k\}$.  By construction of the 2-step nilpotent Lie algebra from the graph, if $|S|>2$, $j(Z_i)$ is singular for all $i=1,\dots,k$.  However, it may be possible, depending on the structure of $G$, to find a $Z\in\zz$ such that $j(Z)$ is nonsingular.  In this case the resulting Lie algebra $\ngg$ is almost nonsingular.  We provide an example of a singular 2-step nilpotent Lie algebra constructed from a graph in Example \ref{singex} and an example of an almost nonsingular 2-step nilpotent Lie algebra constructed from a graph in Example \ref{ansingex}. Proposition \ref{AlmostnonsigularCriterion} gives a necessary and sufficient condition on a graph $G$ in order for $\n_G$ to be almost nonsingular. 

\begin{ex} \label{singex} \textbf{Star Graphs.} Let $K_{1,n}=(S,E)=\left(\{X_i\},\{Z_k\}\right)$, $i=1,\dots,n+1$, $k=1,\dots,n$, be the star graph on $n+1$ vertices, $n>1$, such that $X_1$ is adjacent to $X_i$ for $i=2,\dots ,n+1$ but no other vertices are adjacent. This graph is the complete bipartite graph on $n+1$ vertices partitioned into two sets $\{X_1\}$ and $\{X_2,\dots, X_{n}\}$.

\begin{figure}[h!]
\begin{tikzpicture}[->,>=stealth',shorten >=1pt,auto,
  thick,vertex/.style={circle,draw,fill,scale=.7,font=\sffamily\Large\bfseries},node distance=1.5in,thick]

\node[vertex, label=above:{$X_2$}](X2) {};
  \node[vertex, label=above:{$X_3$}] (X3) [right of =X2] {};
  \node[vertex, label=above:{$X_4$}] (X4)[right of=X3]  {};
\node[draw=none] (a)  [right of=X3] [label= {\hspace{1cm}$\cdots \cdots$}] {};
\node[vertex, label=above:{$X_{n-1}$}](Xn-1)[right of=a]{};
\node[vertex, label=above:{$X_{n}$}] (Xn)[right of = Xn-1]{};

\node[vertex, label=above:{$X_{n+1}$}](Xn+1)[right of = Xn]{};
\node[vertex, label=below:{$X_1$}](X1) [below of = a]{};

  \path[every node/.style={font=\sffamily\small}]
   (X1) edge node[above left] {$Z_1\hspace{.8cm}$}  (X2)
		     edge node [above left]{$Z_2\hspace{.4cm}$}  (X3)
				 edge node [above left] {$Z_3$}  (X4)
				 edge node[above left]{$Z_{n-1}\hspace{-.5cm}$}  (Xn)
		     edge node[above left]{$Z_{n}\hspace{-.6cm}$} (Xn+1)

			edge node [above left] {$Z_{n-2}\hspace{-.3cm}$}  (Xn-1);
\end{tikzpicture}
\caption{Star Graph}
\label{Star Graph} 
\end{figure}
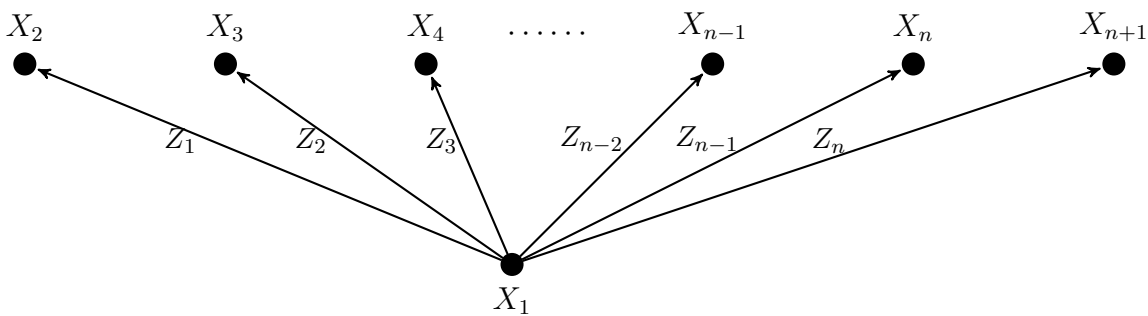

The Lie algebra is defined by the brackets $[X_1,X_i]=Z_{i-1}$ for $i=2,\dots,n+1$.  Clearly, for each $i=1,\dots,n$, $j(Z_i)X_k=0$ for $k=2,\dots,n+1$ except $k=i+1$; thus $j(Z_i)$ is singular for each $i=1,\dots,n$.  Letting $Z=a_1Z_1+\cdots+a_{n}Z_{n}$, we find the matrix representation of $j(Z)$: $$j(Z)=\left(\begin{array}{ccccc} 0 & -a_1 & -a_2 & \cdots & -a_{n} \\ a_1 & 0 & 0 & \cdots & 0 \\ a_2 & 0 & 0 & \cdots & 0 \\ \vdots & \vdots & \vdots & \vdots & \vdots \\ a_{n} & 0 & 0 & \cdots & 0
\end{array}\right).$$ 
From this it is easy to calculate the eigenvalues of $j(Z)$ to be $\{0,\pm i\sqrt{a_1^2+a_2^2+\cdots a_{n}^2}\} = \{0, \pm i |Z|\}$, proving that $j(Z)$ is in fact singular for all $Z\in \zz$ and therefore the associated 2-step nilpotent Lie algebra $\n_{K_{1,n}}$ is also singular.
\end{ex}

\begin{ex} \label{ansingex} \textbf{Even Cycles.} Let $C_{2n}=(S,E)=\left(\{X_i\},\{Z_i\}\right)$,  $i=1,\dots,2n$ be the cycle on $2n$ vertices with $n>1$.  Label the directed graph such that the edge $Z_i$ joins $X_i$ to $X_{i+1}$.  Thus the Lie algebra structure on $\n_{C_{2n}}$ is determined by the brackets $[X_i,X_{i+1}]=Z_i$, $i=1,\dots,2n-1$ and $[X_{2n},X_1]=Z_{2n}$.    Construct a proper spanning subgraph $H=(S_H,E_H)\subset C_{2n}$ such that $S_H=S$ but $E_H=\{Z_i\}\subset E_{C_{2n}}$ for only odd  $i$. 
Then, let $Z=a_1Z_1+a_3Z_3+\cdots+a_{2n-1}Z_{2n-1}$, $a_1, \ldots, a_{2n-1} \in \R$.   Using the matrix representation of $j(Z)$: $$j(Z) = \left(\begin{array}{cccccccc} 0 & a_1 & 0 & 0 & 0 & \cdots & 0 & 0 \\  -a_1 & 0 & 0 & 0 & 0 & \cdots & 0 & 0 \\ 0 & 0 & 0 & a_3 &  0 & \cdots & 0 & 0 \\0 & 0 & -a_3 & 0 & 0 & \cdots & 0 & 0 \\ \vdots & \vdots & 0 & 0 & \ddots & \ddots & \vdots & \vdots \\ 0 & 0 & \cdots & \cdots & \cdots & \cdots & 0 & a_{2n-1} \\ 
0 & 0 & \cdots & \cdots & \cdots & 0 & -a_{2n-1} & 0 \end{array}\right),$$ we find the eigenvalues of $j(Z)$ to be $\{\pm ia_1,\pm ia_3,\dots,\pm ia_{2n-1}\}$.   Then, for any $Z\in\zz$, as above, with $a_i\neq 0$ for all $i$, $j(Z)$ is nonsingular.  Thus $\n_{C_{2n}}$ is almost nonsingular.
\end{ex}

\begin{remark} Any graph $G$ on $2n$ vertices, $n>1$, that contains $C_{2n}$ as a subgraph has an almost nonsingular associated 2-step nilpotent Lie algebra $\ngg$.  To show that $\ngg$ contains a nonsingular $Z\in\zz$, we can use the same construction as in Example \ref{ansingex}. This same technique is used to prove Proposition \ref{AlmostnonsigularCriterion}.
\end{remark}

We now prove a proposition that will allow us to restrict our attention to connected graphs.  It will also help classify all graphs that are associated to almost nonsingular Lie algebras.

\begin{prop}
Let $G=(S,E)$ be a graph with more than one connected component.  Then the associated 2-step nilpotent Lie algebra $\ngg$ is either singular or almost nonsingular.  Further  $\n_G$ is almost nonsingular if and only if each connected component is either nonsingular or almost nonsingular.
\label{almostnsconn}
\end{prop}
\begin{proof}
From Lemma \ref{K2nonsinglemma}, $K_2$ is the only graph associated to a nonsingular 2-step nilpotent Lie algebra.   Thus if $G$ has more than one connected component, the associated Lie algebra is either singular or almost nonsingular.  Let $G=H_1\cup H_2\cup\cdots\cup H_k$ where $H_i$ are the connected components of $G$, $i=1,\dots,k$. The corresponding Lie algebra can be decomposed as $\ngg = \n_{H_1} \oplus \cdots \oplus \n_{H_k}$, where $\n_{H_i}$ is the metric, 2-step nilpotent Lie algebra associated to the graph $H_i$, $1\leq i\leq k$. Each $\n_{H_i}$ can be decomposed as $\n_{H_i} = \V_i \oplus \zz_i$ where $\zz_i$ is the center of $\n_{H_i}$ and $\V_i = \zz_i^{\perp}$. To show that the Lie algebra $\ngg$ is almost nonsingular we need only to find an element $Z$ in the center $\zz$ of $\n_G$ such that $j(Z)$ has purely nonzero eigenvalues.  Note that the eigenvalues of $j(Z)$ will be precisely the union of the eigenvalues of  $j(Z)$ restricted to $\V_i$ for $1\leq i \leq k$.  Assuming each connected component is nonsingular or almost nonsingular, for each $1\leq i \leq k$ choose $Z^*_i$ in the center of $\n_{H_i}$ such that $j(Z^*_i)$ has only nonzero eigenvalues on $\V_i$.  Let $Z=Z^*_1+Z^*_2+\cdots +Z^*_k$.  Then it follows that $j(Z)$ has only nonzero eigenvalues on $\V$ and $\ngg$ must be almost nonsingular.  

Note that if for one connected component $H_l$ of $G$ the subalgebra $\n_{H_l}$ is singular then for any $Z \in \zz$, $j(Z)$ has zero as an eigenvalue. This is true since either $Z$ has a nonzero component in the center of $\n_{H_l}$, hence has zero as an eigenvalue or $Z$ has no component in $\n_{H_l}$, in which case $\V_l$ is contained in the kernel of $j(Z)$. Therefore for any $Z\in \zz$, $j(Z)$ is singular and thus $\ngg$ is singular.  
\end{proof}

Example \ref{ansingex} is typical of graphs for which $\ngg$ is almost nonsingular. The process used in that example of finding a proper spanning subgraph in the graph $C_{2n}$ to demonstrate almost nonsingularity can be generalized to characterize graphs with the property that $\ngg$ is almost nonsingular.

\begin{defn}\label{VertexCovering}
{\rm Let $G=(S,E)$ be a graph with $|S|=2n$, $n>1$ and the set of vertices $S = \{X_1, \ldots, X_{2n}\}$. We say that $G$ has a {\em vertex covering by $n$ disjoint copies of $K_2$} if there exists a permutation $\sigma \in S_{2n}$ such that $X_{\sigma(2i-1)}X_{\sigma(2i)} \in E$ for all $1 \leq i \leq n$, where $ S_{2n}$ is the symmetric group.}
\end{defn}

\begin{ex} The graph $P_4$ and its vertex covering are given in Figure \ref{P4vertcov}.  The permutation $\sigma$ from the definition of the vertex covering is defined $\sigma(1)=1, \ \sigma(2)=4, \ \sigma(3)=2, \ \sigma(4)=3$.
\end{ex}

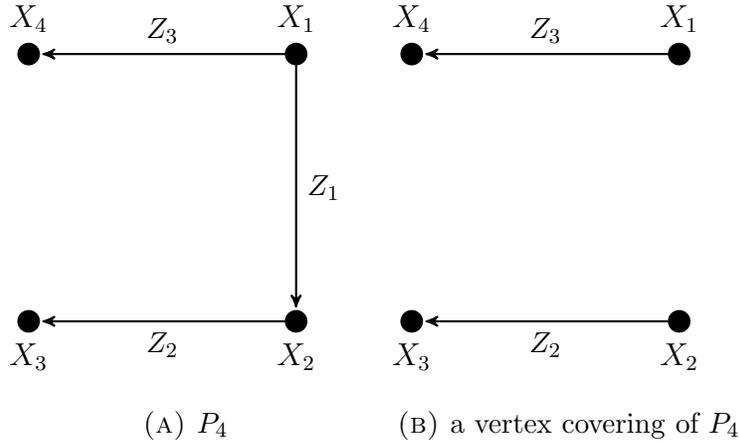
\begin{figure}[h!]

\begin{subfigure}[b]{.3\textwidth}
\begin{tikzpicture}[->,>=stealth',shorten >=1pt,auto,
  thick,vertex/.style={circle,draw,fill,scale=.7,font=\sffamily\Large\bfseries},node distance=2in,thick]

\node[vertex, label=above:{$X_4$}](X4) {};
  \node[vertex, label=below:{$X_3$}](X3) [below of=X4] {};
  \node[vertex, label=above:{$X_1$}] (X1) [right of=X4]  {};
  \node[vertex, label=below:{$X_2$}] (X2)[below of=X1]  {};

\path[every node/.style={font=\sffamily\small}]
 (X1) edge node[midway,right]{$Z_1$} (X2)
        edge node[midway,above]{$Z_3$} (X4)
(X2) edge node[midway,below]{$Z_2$} (X3);

\end{tikzpicture}
\caption{$P_4$}
\end{subfigure}
\begin{subfigure}[b]{.3\textwidth}
\begin{tikzpicture}[->,>=stealth',shorten >=1pt,auto,
  thick,vertex/.style={circle,draw,fill,scale=.7,font=\sffamily\Large\bfseries},node distance=2in,thick]

\node[vertex, label=above:{$X_4$}](X4) {};
  \node[vertex, label=below:{$X_3$}](X3) [below of=X4] {};
  \node[vertex, label=above:{$X_1$}] (X1) [right of=X4]  {};
  \node[vertex, label=below:{$X_2$}] (X2)[below of=X1]  {};

\path[every node/.style={font=\sffamily\small}]

(X2) edge node[midway,below]{$Z_2$} (X3)
 (X1) edge node[midway,above]{$Z_3$} (X4);
\end{tikzpicture}

\caption{a vertex covering of $P_4$}
\end{subfigure}

\caption{An example of a vertex covering}
\label{P4vertcov}
\end{figure}



\begin{prop}\label{AlmostnonsigularCriterion}
 Let $G=(S,E)$ be a graph with $|S|=2n$, $n>1$.  The associated 2-step nilpotent Lie algebra $\n_G$ is almost nonsingular if and only if $G$ has a vertex covering by $n$ disjoint copies of $K_2$.

\end{prop}

\begin{proof}
Suppose that $G$ has a vertex covering by $n$ disjoint copies of $K_2$ and $S =  \{X_1, \ldots, X_{2n}\}$.  Let $\sigma$ denote the permutation  in the symmetric group $S_{2n}$ such that $X_{\sigma(2i-1)}X_{\sigma(2i)} \in E$ for all $1 \leq i \leq n$. We will prove that there exists  a $Z\in\zz$ with all nonzero eigenvalues.
 Let $Z= Z_1+ Z_2+\cdots+ Z_n$ where $Z_i = [X_{\sigma(2i-1)}, X_{\sigma(2i)}]$  for all $1 \leq i \leq n$.
 The matrix of $j(Z)$ with respect to the ordered basis $\{X_{\sigma(1)}, X_{\sigma(2)}, \ldots, X_{\sigma(2n)}\}$ is \[\left(\begin{array}{cccccccc} 0 & 1 & 0 & 0 & 0 & \cdots & 0 & 0 \\  -1 & 0 & 0 & 0 & 0 & \cdots & 0 & 0 \\ 0 & 0 & 0 &1 &  0 & \cdots & 0 & 0 \\0 & 0 & -1 & 0 & 0 & \cdots & 0 & 0 \\ \vdots & \vdots & 0 & 0 & \ddots & \ddots & \vdots & \vdots \\ 0 & 0 & \cdots & \cdots & \cdots & \cdots & 0 & 1 \\
0 & 0 & \cdots & \cdots & \cdots & 0 & -1 & 0 \end{array}\right).\]

Clearly the eigenvalues of $j(Z)$ are  $\pm i$.  Hence $j(Z)$ is nonsingular and $\n_G$ is almost nonsingular.

Now assume that $\n_G$ is almost nonsingular. This means there exists $Z \in \zz$ such that $j(Z)$ is nonsingular. Hence the matrix of $j(Z)$  with  respect  to any ordered basis is invertible. Therefore, take basis $B=\{X_1, \ldots, X_{2n}\}$ for $\V$ where  $\{X_i\}$ is the  set of vertices in the graph $G$.  Denote the matrix of $j(Z)$ in the basis $B$ by  $j(Z)= A=(a_{ij})$.  Since $A$ is a $2n \times 2n$ skew-symmetric matrix,
the determinant of $A$ is the square of the Pfaffian of $A$, denoted by $\pf(A)$. The Pfaffian  is given by
\[\ds \pf(A) =  \frac{1}{2^n n!} \sum_{\sigma \in S_{2n}} \sgn(\sigma) \prod_{i=1}^n a_{\sigma(2i-1) \sigma(2i)}, \]
where $\sgn(\sigma)$ is the signature of $\sigma$  (see, for example, \cite{G}).
Since the determinant of $A$ is nonzero, the $\pf(A)$ is nonzero. That means there exists  a permutation $\sigma \in S_{2n}$ such that
all entries $a_{\sigma(2i-1) \sigma(2i)}$ are nonzero for $1 \leq i \leq n$.  Hence $\langle j(Z)(X_{\sigma(2i)}), X_{\sigma(2i-1)}\rangle \neq 0$ and $X_{\sigma(2i-1)}X_{\sigma(2i)}$ is an edge in the graph $G$ for all $1 \leq i \leq n$ which means that $G$ has a vertex covering by $n$ disjoint copies of $K_2$.
\end{proof}

\section{Heisenberg-like Lie algebras}\label{Heisenbergsect}

There are many different equivalent characterizations of Heisenberg-like Lie algebras.  Gornet and Mast originally defined them as generalizations of Heisenberg type Lie algebras using a totally geodesic condition, see \cite{GM}.  We will use a characterization given by Gornet and Mast, which is stated in Theorem \ref{GMHlikeThm} below, as the defining property of Heisenberg-like 2-step nilpotent Lie algebras.  Denote the distinct eigenvalues of $j(Z)$ by $\{\pm i \vartheta_1(Z), \pm i \vartheta_2(Z), \dots, \pm i \vartheta_m(Z)\}$ for any $Z\in\zz$.  Note that for each Heisenberg-like 2-step nilpotent Lie algebra $m$ is a fixed integer by the following result.  It is an immediate corollary that all Heisenberg-like 2-step nilpotent Lie algebras are either singular or nonsingular.

\begin{thm}[\cite{GM} Theorem 3.7]  A 2-step nilpotent metric Lie algebra is Heisenberg-like if and only if for every $i=1,\dots,m$ there is a constant $c_i\geq 0$ such that for every nonzero $Z\in\zz$, $\vartheta_i(Z)=c_i|Z|$.
\label{GMHlikeThm}
\end{thm}

The following proposition characterizes graphs that are associated to Heisenberg-like Lie algebras.

\begin{prop}  If $\ngg$ is a Heisenberg-like Lie algebra associated with a graph $G$, then the graph $G$ does not contain a path of length three on four distinct vertices.
\label{path3}
\end{prop}
\begin{proof}
Assume that $G$ has a path of length three on four distinct vertices.  Label the vertices and edges of one such path so that the bracket relations are given by $[X_1,X_2]=Z_1, \ [X_2,X_3]=Z_2, \ [X_3,X_4]=Z_3$.  Assume that $\{Z_i\}$ is part of an orthonormal basis on $\zz$ and consider the subalgebra of $\ngg$ with basis $\{X_1, X_2, X_3, X_4, Z_1, Z_2, Z_3\}$.  In this subalgebra, consider $j(Z_1+Z_3)$ with matrix representation $\left( \begin {array}{cccc} 0&1&0&0
\\ -1&0&0&0\\ 0&0&0&1
\\0&0&-1&0\end {array}  \right) 
$.  Clearly the eigenvalues of $j(Z_1+Z_3)$ are $\{\pm i\}$, each with multiplicity two, and $|Z_1+Z_2|=\sqrt{2}$.  Thus if $\vartheta_i(Z)=c_i|Z|$, $\ds c_1=c_2=\frac{1}{\sqrt{2}}$.  Next consider $j(Z_1+Z_2+Z_3)$ with matrix representation $\left(   \begin {array}{cccc} 0&1&0&0
\\ -1&0&1&0\\ 0&-1&0&1
\\ 0&0&-1&0\end {array}   \right) $.  Here the eigenvalues are $\{\pm \frac{1}{2}i\left(\sqrt{5}-1\right), \pm \frac{1}{2}i\left(\sqrt{5}+1\right)\}$ but $|Z_1+Z_2+Z_3|=\sqrt{3}$.  In this case we found $\ds c_1=\frac{\sqrt{5}-1}{2\sqrt{3}},   c_2=\frac{\sqrt{5}+1}{2\sqrt{3}}$.  Therefore we have produced two different elements of $\zz$ for which we cannot find common constants such that the characterization of Heisenberg-like Lie algebras hold.  Thus $\ngg$ cannot be Heisenberg-like if it is associated to a graph $G$ containing a path of length three on four distinct vertices.

\end{proof}

\begin{cor}
A 2-step nilpotent Lie algebra $\ngg$ associated with a graph $G$ is Heisenberg-like if and only if one of the following holds:
\begin{enumerate}
\item $G$ is a star graph $K_{1,n}$ for $n>1$.
\item $G$ is the complete graph $K_3$.
\item $G$ is the disjoint union of a connected graph given in (1) or (2) together with with $m$ isolated vertices for some $m\in \N$.
\end{enumerate}
\label{hlike}
\end{cor}

\begin{proof}

Assume that $G$ is a connected graph. If $\ngg$ is Heisenberg-like, it must be associated to a graph $G$ that does not contain a path of length three on four distinct vertices, by Proposition \ref{path3}.  Thus $G$ must be a star or $K_3$. 

Conversely, let $G=K_{1,n}$ be a star graph as in Example \ref{singex}. For any nonzero $Z\in \zz$, the eigenvalues of $j(Z)$ are $\{0,\pm i |Z|\}$ with zero having multiplicity $n-1$.  Thus $\n_{K_{1,n}}$ is Heisenberg-like.  

Next let  $G=K_3$ with vertices $\{X_1, X_2, X_3\}$ and edges $\{Z_1, Z_2, Z_3\}$. Label the edges so that in $\ngg$ we have the bracket relations $[X_1,X_2]=Z_1, \ [X_2,X_3]=Z_2, \ [X_1,X_3]=Z_3$.  Letting $Z=a_1Z_1+a_2Z_2+a_3Z_3$, $a_1, a_2, a_3 \in \R$, then in the standard basis for $\V$,  
$j(Z)=\left(   \begin {array}{ccc} 
0 & -a_1 & -a_3\\
  a_1&0& -a_2\\ 
a_3&a_2&0
\end{array} \right) 
$.  The map $j(Z)$ has eigenvalues $\{0,\pm i \sqrt{a_1^2+a_2^2+a_3^2}\}=\{0,\pm i |Z|\}$.  Thus $\n_{K_3}$ is also Heisenberg--like.

Suppose $G$ has more than one connected component.  Then $\ngg$ is either almost nonsingular or singular.  If $\ngg$ is almost nonsingular it is not Heisenberg-like.  So suppose $\ngg$ is singular and there are two distinct edges in different connected components of $G$. If $Z_1$, $Z_2$ correspond to edges in different connected components of $G$ then the multiplicity of the zero eigenvalue of $j(Z_1)$ will be strictly larger than the multiplicity of the zero eigenvalue of $j(Z_1 + Z_2)$. Since the multiplicity of the zero eigenvalue changes depending on $Z$, we conclude that $\ngg$ is not Heisenberg-like.

Let $G$ be a connected graph with corresponding metric two-step nilpotent Lie algebra $\n_G$ and  let $G^*$ be the disjoint union of $G$ together with $m$ isolated vertices. Denote the corresponding metric 2-step nilpotent Lie algebra by $\n_{G^*}$ and let $i:\n_G \hookrightarrow \n_G^*$ be the inclusion isometry. Every nonzero element $Z^*$ in the center of $\n_{G^*}$ is equal to $i(Z)$ for some $Z$ in the center of $\n_G$. The eigenvalues of $j(Z^*)$ are the same as the eigenvalues of $j(Z)$, increasing the multiplicity of the zero eigenvalue by $m$. Therefore, $\n_{G^*}$ is Heisenberg-like if and only if $\n_G$ is Heisenberg-like.

\end{proof}

\begin{remark} The Lie algebra associated to the star graph is precisely the first case in Example 6.3 of \cite{DDM}. \end{remark}

\section{density of closed geodesics results}\label{DCGsect}

We continue the study begun in \cite{E} and continued in \cite{M}, \cite{LP}, \cite{DeM}, and \cite{DC}, of the density of closed geodesics property of the quotient of a simply connected, 2-step nilpotent Lie group $N$ by a lattice $\Gamma \subseteq N$.  The investigation into the density of closed geodesics property relies heavily on the geodesic equations, which were solved in \cite{E} for a simply connected, metric 2-step nilpotent Lie group. These equations are given below in Proposition \ref{geoeqns}. This section begins with terminology and several essential previously known results and ends with fundamental questions which are partially addressed in the subsequent results of this paper.

\begin{defn} {\rm {Let $M$ be a compact manifold with Riemannian metric. The manifold $M$ is said to have the \textit{density of closed geodesics property} if the vectors tangent to closed  geodesics are dense in the unit tangent bundle $TM$. 
}}\end{defn}

\begin{defn} {\rm {Given a metric 2-step nilpotent Lie algebra $\metn$, a nonzero element $Z \in \zz$ is said to be \textit{in resonance} if the ratio of any pair of nonzero eigenvalues of $j(Z)$ is rational.} }
\end{defn}

The resonance condition is a metric Lie algebra condition. When convenient, say that the map $j(Z)$ is in resonance rather than the element $Z$.

In \cite{Ma}, Mast gives a necessary condition for the density of closed geodesics in $\GN$. Note that the theorem in \cite{Ma} is stated for nonsingular $N$, but the proof given there holds in general. 

\begin{thm} \cite{Ma} Let $N$ be a simply connected, 2-step nilpotent Lie group with a left invariant metric. If $\GN$ has the density of closed geodesics property for some lattice $\Gamma$ in $N$, then $Z$ is in resonance for a dense subset of $Z$ in $\zz$.
\label{MaResIsNec}
\end{thm}

Mast proved that $\GN$ has the density of closed geodesic property for every lattice $\Gamma \subseteq N$ under the conditions that $\metn$ is nonsingular and that every nonzero $Z\in \zz$ is in resonance.  

\begin{thm}\cite{Ma}\label{MaDCG}  Let $N$ be a
  nonsingular, simply 
  connected, 2-step 
  nilpotent Lie group  with a  left invariant metric.  Suppose $j(Z)$ is in resonance for all nonzero $Z \in
  \zz$. Then for any lattice $\Gamma$ in $N$, $\GN$ has the density of
  closed geodesics property. 
\end{thm}

Lee and Park generalized Mast's result,   relaxing both  the nonsingularity and the resonance hypotheses.

\begin{thm}\cite{LP}\label{LPDCG} Let $\{ \n,
  \metric \}$ be a  metric, 2-step 
  nilpotent Lie algebra such that 
\begin{enumerate}
\item $j(Z)$ is in resonance for a dense set of $Z \in \zz$.
\item $j(Z)$ is nonsingular for some nonzero $Z \in
  \zz$.  
\end{enumerate} 
Then for any lattice $\Gamma$ of $N$, the nilmanifold $\GN$ has the
density of closed geodesics property.  
\end{thm}

The results on the density of closed geodesics above and the results given later in this section rely on the following two propositions from \cite{E}. Proposition \ref{EResonanceProp} provides a strong connection between resonance and the geodesics in $N$ which are translated by left multiplication of a lattice element. This provides information on geodesics which project to smoothly closed geodesics in the quotient space $\GN$. 

Proposition \ref{geoeqns}  gives explicit solutions to the geodesic equations. Proposition \ref{geoeqns} requires some additional preliminary notation, decomposing $\V = \zz^{\perp}$ into invariant subspaces of $j(Z)$ for some $Z\in \zz$. These equations and the decomposition will be used again in subsection \ref{subsec:FirstHitApproach}.

\begin{prop}(\cite{E} Proposition 4.3 and Lemma 4.13) Let $N$ be a 2-step nilpotent Lie group with Lie algebra $\n$. Let $Z\in \zz$ be nonzero. The element $Z$ is in resonance if and only if $e^{\om j(Z)} = Id$ on $\V$ for some $\om >0$. Furthermore, $Z$ in resonance implies that the geodesic $\gamma$ in $N$ with $\gamma(0)=e$ and $\gamma'(0) = X+Z$ is translated by $\phi = \gamma(\om)$, that is, $\phi \cdot \gamma(t) = \gamma(t+\om)$,  where $\om>0$ satisfies $e^{\om j(Z)} = Id$ on $\V$.
\label{EResonanceProp}
\end{prop}

Let $\n$ be a  2-step nilpotent Lie algebra with left invariant metric and write $\n =\V\oplus \zz$ as usual.  Let $Z$ be any nonzero element of $\zz$. Let $\{\pm i \vartheta_1(Z),\ldots , \pm i \vartheta_m(Z)\}$
denote the distinct nonzero eigenvalues of $j(Z)$. Then we decompose $\V$  as $\ds \V =
\oplus_{k=0}^m W_k$ where $W_0 = \ker (j(Z))$ and for $k\in\{1,2,\ldots, m\}$, $W_k$ is the eigenspace of $j(Z)^2$
corresponding to eigenvalue $-\vartheta_k(Z)^2$. Thus, for any $X\in \V$ we can write
\vspace{0.2in}

\begin{eqnarray}
X &=& V_1 +V_2 \mbox{ where } V_1\in W_0 \mbox{ and
  }\nonumber \\ 
V_2& =& \sum_{k=1}^{m} \zeta_k, \mbox{ where } \zeta_k \in W_k \mbox{ for }1
\leq k \leq m. \nonumber
\end{eqnarray}

The following proposition gives the solutions to the geodesic equations.

\begin{prop}\label{geoeqns}\cite{E} Let $\{N,\metric\}$ be a simply
  connected 2-step nilpotent Lie group with a left invariant metric
  and with Lie algebra $\n= \V \oplus \zz$. Let $\gamma(t) $ be a
  geodesic with $\gamma(0)=e$. Write $\gamma'(0) = X +Z \in \n$
  where $X \in \V$ and $Z\in \zz$. Write $\gamma(t) = \exp
  (X(t)+Z(t))$ where $X(t)\in\V$ and $Z(t)\in \zz$ for all $t\in \R$
  and $X'(0) = X$ and $Z'(0) = Z$. Then with respect to the
  notation above we have
\begin{enumerate}
\item $X(t) = tV_1+\left( e^{tj(Z)} - Id \right)(j(Z)^{-1}V_2)$.
\item $Z(t) = t\widetilde{Z}_1+\widetilde{Z}_2$, where
\begin{enumerate}
\item $\ds \widetilde{Z}_1 = Z+\frac{1}{2}[V_1,
  (e^{tj(Z)}+Id)(j(Z)^{-1}V_2)]+ \frac{1}{2}\sum_{k=1}^m
  [j(Z)^{-1}\zeta_k, \zeta_k]$
\item\begin{eqnarray} \widetilde{Z}_2(t) & = & [V_1,
    (Id-e^{tj(Z)})j(Z)^{-2}V_2]
    +\frac{1}{2}[e^{tj(Z)}j(Z)^{-1}V_2,j(Z)^{-1}V_2]\nonumber \\
    & & -\frac{1}{2}\sum_{i \neq k = 1}^m \left(\frac{1}{\vartheta_k(Z)^2 -
      \vartheta_i(Z)^2}\right)\left\{[e^{tj(Z)} j(Z)\zeta_i, e^{tj(Z)}
    j(Z)^{-1}\zeta_k] - [e^{tj(Z)}\zeta_i,e^{tj(Z)}\zeta_k]\right\}
    \nonumber
    \\
    & & +\frac{1}{2} \sum_{i \neq k = 1}^m \left(\frac{1}{\vartheta_k(Z)^2 -
      \vartheta_i(Z)^2}\right)\left\{[j(Z)\zeta_i, j(Z)^{-1} \zeta_k] - [\zeta_i,
    \zeta_k]\right\} \nonumber
\end{eqnarray}
\end{enumerate}
\end{enumerate}
\label{GeodesicEquations}
\end{prop}

There are several fundamental questions which arise in studying the density of closed geodesics. Given any simply connected 2-step nilpotent Lie group $N$ with any left invariant metric and given any lattice $\Gamma \subseteq N$ will the quotient $\GN$ have the density of closed geodesics property? The answer to this question is no, in general. We cannot guarantee density of closed geodesics in every $N$, as the condition of resonance is necessary, as given by Mast in Theorem \ref{MaResIsNec} above. 

Restricting to simply connected, metric, 2-step nilpotent Lie groups $N$ which do satisfy the resonance condition, the question of density of closed geodesics has been partially answered. The answer is affirmative in the case of Lie groups whose Lie algebra is either nonsingular or almost nonsingular by Lee-Park's result, Theorem \ref{LPDCG}. 

In the case of 2-step nilpotent Lie groups with singular Lie algebras, only partial results are known. In \cite{DeM}, an infinite class of singular, 2-step nilpotent, metric Lie algebras is given for which the compact quotient $\GN$ satisfies the density of closed geodesics property for any lattice $\Gamma \subseteq N$. In \cite{DC}, an infinite class of singular, 2-step nilpotent, metric Lie algebras is given for which the compact quotient $\GN$ satisfies the density of closed geodesics property for a certain large class of lattices $\Gamma \subseteq N$. There is no known example of a simply connected, singular, 2-step nilpotent Lie group with lattice $\Gamma  \subseteq N$ which satisfies the necessary resonance condition but which fails to have the density of closed geodesics in the quotient $\GN$. 

In the remainder of this paper, we address the density of closed geodesics in simply connected, metric, 2-step nilpotent Lie groups whose Lie algebra is constructed from a graph. Beginning in subsection \ref{k4graphs}, we apply the result of Lee-Park on almost nonsingular Lie algebras to establish density of closed geodesics in $\GN$ for any lattice $\Gamma$ in groups $N$ whose Lie algebra is constructed from particular connected graphs on four vertices. In subsection \ref{subsec:FirstHitApproach}, the first hit map method introduced by Eberlein in \cite{E} and used again in \cite{M},\cite{LP}, \cite{DeM}, and \cite{DC}, is explained and applied in the case $\n = \n_G$ where $G=P_3$ is the path on three vertices. It is proven using this method that the simply connected Lie group corresponding to the  singular, metric, 2-step nilpotent Lie algebra $\n_{P_3}$ has the density of closed geodesics property for any lattice $\Gamma \subseteq N$. 

The established first hit map method fails in the case of $\n_G$ where $G$ is a star graph, $K_{1,n}$, and in the case where $G$ is the complete graph on three vertices, $K_3$. These cases are addressed in subsections \ref{stargraphs} and \ref{K3graphSubsection}. In each case, a direct method is used to produce a particular lattice for which $\GN$ will have the density of closed geodesics.

\subsection{Density of closed geodesics for graphs on four vertices} \label{k4graphs}

We prove that 2-step nilmanifolds arising from connected graphs  on four vertices have the density of closed geodesics property.  The proof of this result relies on the density of resonant vectors in the center of the associated Lie algebra.  

\begin{prop}  Let $\n_G=\V\oplus\zz$ be the 2-step nilpotent metric Lie algebra constructed from a connected graph $G$ on four vertices. There is a dense subset of $Z\in\zz$ such that $Z$ is in resonance.  \label{k4res} 
\end{prop}

The proof of Proposition \ref{k4res} is largely computational and is found in Section \ref{pfk4res}.  We use Theorem \ref{LPDCG}, to prove directly Corollary \ref{Cor:DCGPK4}, namely the density of closed geodesics property for any almost nonsingular nilmanifold associated to a connected graph on four vertices.




Let $K_4$ be the complete graph on four vertices and label the vertices $\{X_1,X_2,X_3,X_4\}$ and edges $\{Z_1,Z_2,\dots,Z_6\}$ as in Figure \ref{k4fig}.

\begin{figure}[h!]

\begin{tikzpicture}[->,>=stealth',shorten >=1pt,auto,
  thick,vertex/.style={circle,draw,fill,scale=.7,font=\sffamily\Large\bfseries},node distance=2in,thick]

\node[vertex, label=above:{$X_4$}](X4) {};
  \node[vertex, label=below:{$X_3$}](X3) [below of=X4] {};
  \node[vertex, label=above:{$X_1$}] (X1) [right of=X4]  {};
  \node[vertex, label=below:{$X_2$}] (X2)[below of=X1]  {};

\path[every node/.style={font=\sffamily\small}]

(X1) edge node[midway,right]{$Z_1$} (X2)
       edge node[midway,above]{$Z_3$} (X4)
        edge node[near end, below right]{\hspace{-.2cm}$Z_2$} (X3)
 (X2) edge node[midway,below]{$Z_4$} (X3)
       edge node[near end, below right]{\hspace{-.4cm}$Z_5$} (X4)
 (X3) edge node[midway,left]{$Z_6$} (X4);

\end{tikzpicture}
\caption{$K_4$}
\label{k4fig}
\end{figure}
%
Then the Lie algebra constructed from $K_4$ has bracket relations: $[X_1,X_2]=Z_1, \ [X_1,X_3]=Z_2, \ [X_1,X_4]=Z_3, \ [X_2,X_3]=Z_4, \  [X_2,X_5]=Z_5,$ and $ [X_3,X_4]=Z_6$.  

The subgraphs of $K_4$ with edge and vertex labels are given in Figure \ref{figsubk4}.  
The underlying undirected graphs are the isomorphism classes of  all the connected subgraphs of $K_4$ on four vertices.

\begin{figure}[h!]

\begin{subfigure}[b]{.3\textwidth}
\begin{tikzpicture}[->,>=stealth',shorten >=1pt,auto,
  thick,vertex/.style={circle,draw,fill,scale=.7,font=\sffamily\Large\bfseries},node distance=2in,thick]

\node[vertex, label=above:{$X_4$}](X4) {};
  \node[vertex, label=below:{$X_3$}](X3) [below of=X4] {};
  \node[vertex, label=above:{$X_1$}] (X1) [right of=X4]  {};
  \node[vertex, label=below:{$X_2$}] (X2)[below of=X1]  {};

\path[every node/.style={font=\sffamily\small}]
 (X1) edge node[midway,right]{$Z_1$} (X2)
       edge node[midway,above]{$Z_3$} (X4)
        edge node[near end,below right]{\hspace{-.4cm} $Z_2$} (X3)
 (X2) edge node[midway,below]{$Z_4$} (X3)
       edge node[near end,below left]{\hspace{-.4cm} $Z_5$} (X4);

\end{tikzpicture}

\caption{$G_1$}
\end{subfigure}
~
\begin{subfigure}[b]{.3\textwidth}
\begin{tikzpicture}[->,>=stealth',shorten >=1pt,auto,
  thick,vertex/.style={circle,draw,fill,scale=.7,font=\sffamily\Large\bfseries},node distance=2in,thick]

\node[vertex, label=above:{$X_4$}](X4) {};
  \node[vertex, label=below:{$X_3$}](X3) [below of=X4] {};
  \node[vertex, label=above:{$X_1$}] (X1) [right of=X4]  {};
  \node[vertex, label=below:{$X_2$}] (X2)[below of=X1]  {};

\path[every node/.style={font=\sffamily\small}]
 (X1) edge node[midway,right]{$Z_1$} (X2)
       edge node[midway,above]{$Z_3$} (X4)
        edge node[near end,below right]{$Z_2$} (X3)
 (X2) edge node[midway,below]{$Z_4$} (X3);

\end{tikzpicture}

\caption{$G_2$}
\end{subfigure}
~
\begin{subfigure}[b]{.3\textwidth}
\begin{tikzpicture}[->,>=stealth',shorten >=1pt,auto,
  thick,vertex/.style={circle,draw,fill,scale=.7,font=\sffamily\Large\bfseries},node distance=2in,thick]

\node[vertex, label=above:{$X_4$}](X4) {};
  \node[vertex, label=below:{$X_3$}](X3) [below of=X4] {};
  \node[vertex, label=above:{$X_1$}] (X1) [right of=X4]  {};
  \node[vertex, label=below:{$X_2$}] (X2)[below of=X1]  {};

\path[every node/.style={font=\sffamily\small}]
 (X1) edge node[midway,right]{$Z_1$} (X2)
       edge node[midway,above]{$Z_3$} (X4)
 (X2) edge node[midway,below]{$Z_4$} (X3)
(X3)edge node[near end,below left]{$Z_6$} (X4);

\end{tikzpicture}
\caption{$C_4$}
\end{subfigure}
\\
\begin{subfigure}[b]{.3\textwidth}
\begin{tikzpicture}[->,>=stealth',shorten >=1pt,auto,
  thick,vertex/.style={circle,draw,fill,scale=.7,font=\sffamily\Large\bfseries},node distance=2in,thick]

\node[vertex, label=above:{$X_4$}](X4) {};
  \node[vertex, label=below:{$X_3$}](X3) [below of=X4] {};
  \node[vertex, label=above:{$X_1$}] (X1) [right of=X4]  {};
  \node[vertex, label=below:{$X_2$}] (X2)[below of=X1]  {};

\path[every node/.style={font=\sffamily\small}]
 (X1) edge node[midway,right]{$Z_1$} (X2)
       edge node[midway,above]{$Z_3$} (X4)
 (X2) edge node[midway,below]{$Z_4$} (X3);

\end{tikzpicture}
\caption{$P_4$}
\end{subfigure}
~
\begin{subfigure}[b]{.3\textwidth}
\begin{tikzpicture}[->,>=stealth',shorten >=1pt,auto,
  thick,vertex/.style={circle,draw,fill,scale=.7,font=\sffamily\Large\bfseries},node distance=2in,thick]

\node[vertex, label=above:{$X_4$}](X4) {};
  \node[vertex, label=below:{$X_3$}](X3) [below of=X4] {};
  \node[vertex, label=above:{$X_1$}] (X1) [right of=X4]  {};
  \node[vertex, label=below:{$X_2$}] (X2)[below of=X1]  {};

\path[every node/.style={font=\sffamily\small}]
 (X1) edge node[midway,right]{$Z_1$} (X2)
       edge node[midway,above]{$Z_3$} (X4)
        edge node[near end,below right]{$Z_2$} (X3);

\end{tikzpicture}
\caption{$K_{1,3}$}
\end{subfigure}

\caption{Connected subgraphs of $K_4$ on four vertices}
\label{figsubk4}
\end{figure}
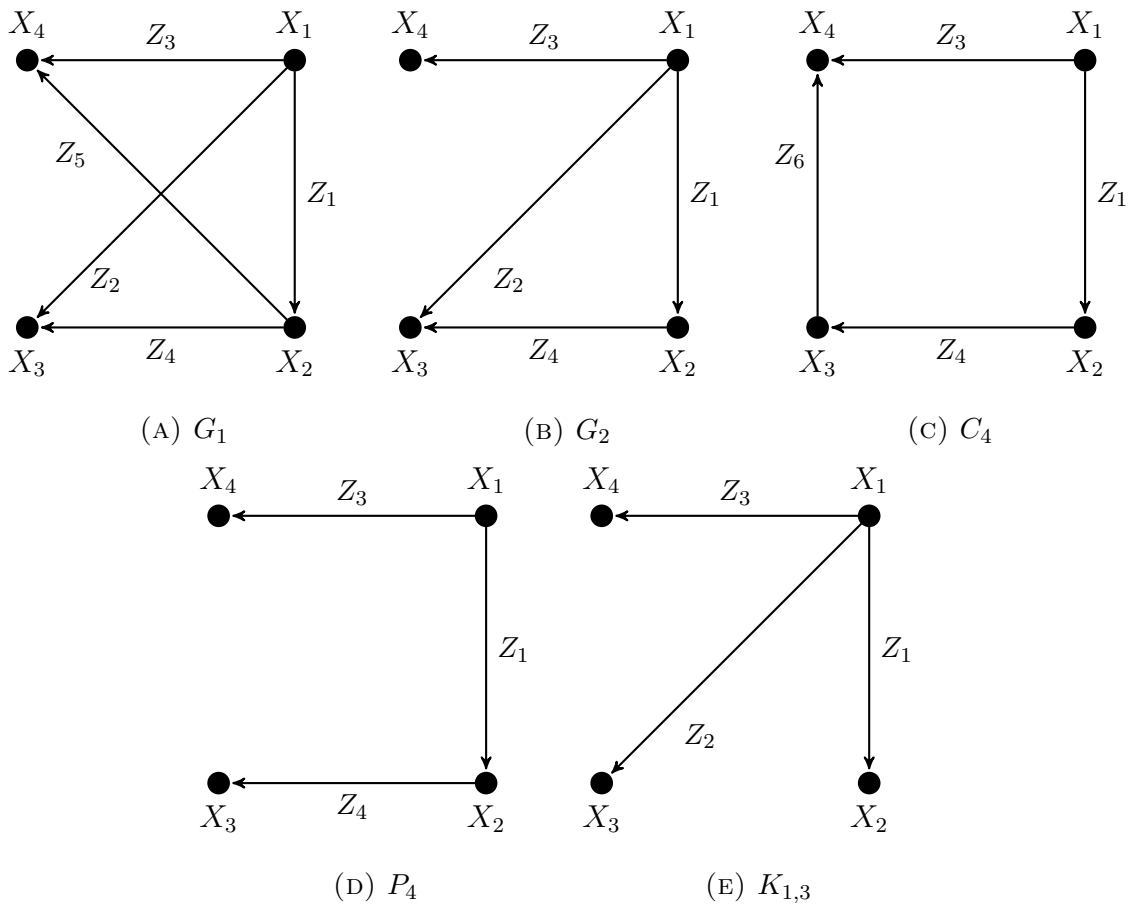


The Lie algebra bracket definitions for each of the proper subgraphs follow from the brackets of $K_4$.  For any edge in $K_4$ that is removed to obtain the subgraph $G$,  the corresponding bracket in the corresponding Lie algebra $\n_G$ will be zero.


\begin{lemma}
The 2-step nilpotent Lie algebra $\n_G$ is  almost nonsingular when $G$ is $K_4$, $G_1$, $G_2$, $C_4$ or $P_4$.
\end{lemma}
\begin{proof}
Let $G$ be the graph $K_4$, $G_1$, $G_2$, $C_4$ or $P_4$ and let $\n_G=\V\oplus\zz$ be the associated 2-step nilpotent Lie algebra.  We showed in Section \ref{singsect} that every 2-step nilpotent Lie algebra associated to any graph other than $K_2$ is either singular or almost nonsingular.    In the proof of Proposition \ref{path3},  an element $Z\in\zz$ is constructed such that $j(Z)$ has four distinct nonzero eigenvalues for any graph $G$ containing a path of length three on four vertices.  Note that each $G$ considered here contains a path of length three on four distinct vertices. Also, since $\dim(\V)=4$ there are precisely four eigenvalues in each case. Since all of the eigenvalues of the constructed $Z$ are nonzero, $Z$ is a nonsingular element of $\zz$ for the respective Lie algebra and $\ngg$ is almost nonsingular.
\end{proof}

By Theorem \ref{LPDCG}, almost nonsingular nilmanifolds have the DCG property if they have a dense set of resonant vectors.  Proposition \ref{k4res} shows that Lie algebras associated to all connected subgraphs  of $K_4$ on four vertices have a dense set of resonant vectors.  The following corollary follows immediately.

\begin{cor}
Let $\ngg$ be the 2-step nilpotent metric Lie algebra constructed from a connected subgraph $G$ of  $K_4$ where $G$ is one of the following $K_4$, $G_1$, $G_2$, $C_4$ or $P_4$.  Let $N$ be the associated simply connected 2-step nilpotent Lie group.  For any lattice $\Gamma$ in $N$, the manifold $\GN$ has the density of closed geodesics property. 
\label{Cor:DCGPK4} 
\end{cor}

Note that there are two cases of connected subgraphs of $K_4$ on four vertices: the almost nonsingular Lie algebras, as indicated above and included in this result, and the singular Lie algebra associated to the star graph $K_{1,3}$.  The proof of the DCG property on the nilmanifold associated to the star graph is addressed in Section \ref{stargraphs}.  


\subsection{The First Hit Map Approach to the Density of Closed Geodesics}\label{subsec:FirstHitApproach} Let $\zz_R$ denote the set of vectors $Z\in \zz$ for which $j(Z)$ is in resonance. Assume that $\zz_R$ is a dense open subset of $\zz$ and let $\Gamma$ be any lattice in $N$.

For any nonzero $\xi \in \n$, let $\gamma_\xi(t)$ denote the geodesic in $N$ with $\gamma_\xi (0) = e$ and $\gamma'(0)=\xi$.  The geodesic $\gamma_\xi (t)$ is translated by $g\in N$ if $g\cdot \gamma_\xi (t) = \gamma_\xi (t+\om)$ for some $\om>0$ and for all $t\in \R$.     If the geodesic $\gamma_\xi(t)$ is translated by a lattice element $g\in \Gamma$, then the geodesic $\gamma_\xi(t)$ in $N$ projects to a smoothly closed geodesic in $\GN$. By Lemma 13 of \cite{DeM}, if vectors tangent to smoothly closed geodesics in $\GN$ are dense in $\n$ then such vectors are dense in $T(\GN)$.  
Hence,  it is enough to only consider geodesics which pass through $e\in N$. 

For any nonzero element $Z\in \zz$, define the following subspaces of $\n$. 

$\w_{Z} = \zz \oplus \ker j(Z) $

$\u_Z = \{X + r Z \, | \, r\neq 0, X \mbox { has nonzero component in } \ker j(Z)\}$\\

 The Lie algebra $\n$  is said to have a rational structure if there is a rational Lie subalgebra $\n_\Q$ with $\n=\n_\Q \otimes \R$. Equivalently, if  $\{ \zeta_1, \ldots , \zeta_n \}$ is a basis for $\n$ over $\R$, then $\n_\Q = \spn_\Q \{\zeta_1, \ldots, \zeta_n\}$ is a rational structure if and only if the basis has rational structure constants. If $\Gamma$ is a lattice in $N$ then $\n_{\Q} = \spn_{\Q} \{\log \Gamma\}$ gives a rational structure on $\n$.  A subalgebra $\h \subseteq \n$ is said to be rational with respect to a lattice $\Gamma$ in $N$ if the subalgebra $\h$ is rational with respect to the rational structure $\spn_{\Q} \{\log \Gamma\}$. If  $\Gamma$ is a lattice in $N$, then $\h \cap \log \Gamma$ is a vector lattice in $\h$ if and only if $\h$ is a rational subalgebra of $\n$. See \cite{CG}. 

For any lattice $\Gamma \subseteq N$, define the set $\zz_\Gamma$ in $\zz$ as follows.

$\zz_\Gamma = \{ Z \in \zz | \w_{Z} = \zz \oplus \ker j(Z) \mbox{ is rational with respect to  } \Gamma\}$.  \\

\begin{remark} 

\begin{enumerate} 

\item By Lemma 16 of \cite{DeM}, the set $\zz_\Gamma$ is dense in $\zz$ for any lattice $\Gamma \subseteq N$. If we assume that $\zz_R$ is open and dense in $\zz$, then the set $\zz_\Gamma\cap\zz_R$ is dense in $\zz$. 
\item $\w_Z$ is a subalgebra of $\n$ and $W_Z = exp(\w_Z)$ is a subgroup of $N$. 
\item $\w_{rZ} = \w_{Z}$ for all nonzero $r\in \R$. 
\item An element $Z \in \zz_\Gamma$ if and only if $rZ \in \zz_\Gamma$ for all nonzero $r\in \R$. 

\end{enumerate}
\end{remark}

 Let $\xi = X+rZ \in \n$ where $rZ\in \zz_R\cap \zz_\Gamma$,  $r\neq 0$ in $\R$ and $X$ has a nonzero component in the kernel of $j(Z)$. That is, $X+rZ\in \u_Z$. Since $Z$ is in resonance there exists a positive real number $\om$ so that $e^{\om j(rZ)} = Id$ on $\V$ by Proposition \ref{EResonanceProp}. Also by Proposition \ref{EResonanceProp}, it follows that $g =\gamma_\xi(\om)$ translates $\gamma_\xi(t)$. That is $\gamma_\xi(\om) \cdot \gamma_\xi (t) = \gamma_\xi (t+\om)$ for all $t\in \R$.  Using the geodesic equations given in Proposition \ref{geoeqns}, and the fact that $e^{\om j(rZ)} = Id$ on $\V$, it follows that $\log(\gamma_\xi(\om)) \in \w_{Z}$ .

Define a ``first hit'' map $F_Z: \u_Z \rightarrow \w_Z$ by $F(\xi) = \log (\gamma_\xi(\om))$. The map $F_Z$ carries the initial velocity of the geodesic $\gamma_\xi$ to the log of the intersection of $\gamma_\xi$ with $W_Z$, where $W_Z = \exp(\w_Z)$. 

From the geodesic equations, it is apparent that the geodesic intersects the subgroup $W_Z$ periodically, specifically for $t=m\om$ for any positive integer $m$. Therefore,  define the $m^{th}$ hit $F_Z^m : \u_Z \rightarrow \w_Z$ by $F_Z^m(\xi) = \log (\gamma_\xi (m\om))$. We note that
\begin{eqnarray*}
F_Z^m(\xi) &=& \log (\gamma_\xi (m\om)) \\
  &= & \log((\gamma_\xi(\om))^m)\\
	&=& m\log (\gamma_\xi(\om))\\
	&=& mF_Z(\xi).
\end{eqnarray*}

Assume the map $F_Z$ has maximal rank on a dense open subset $\widehat{\u}_Z$ of $\u_Z$. Then for any open set $\O \subseteq \widehat{\u}_Z\subseteq \u_Z$, the set $F_Z(\O)$ contains an open set $\U \subseteq F_Z(\O) \subseteq \w_Z$. Since $F_Z^m (\xi) = mF_Z(\xi)$, the set  $F_Z^m(\O)$ contains the set $m\U$. Therefore, for $m$ sufficiently large, $m\U \cap \log \Gamma \neq \emptyset$. 

If $Z\in \zz_\Gamma$ then $\w_Z$ is a rational subalgebra of $\n$ with respect to the lattice $\Gamma$. Therefore $\Gamma \cap W_Z$ is a lattice in $W_Z = exp (\w_Z) \subseteq N$.  Hence, there is a vector $\nu \in \O$ such that $\exp(F_Z^m (\nu)) = \gamma_\nu (m\om) \in \Gamma$. 

Thus $\gamma_\nu$ is a geodesic which is translated by an element of the lattice and $\gamma_\nu$ projects to a smoothly closed geodesic in $\GN$. Since $\O \subseteq \widehat{\u}_Z \subseteq \u_Z$ is an arbitrary open set, we conclude that the set of tangent vectors $\nu \in \u_Z$ for which $\gamma_\nu$ projects to a smoothly closed geodesic is dense in $\u_Z$. The union $\bigcup_{Z\in \zz_\Gamma \cap \zz_R} \u_Z$ is dense in $\n$. Hence $\GN$ has the density of closed geodesics property. This proves the following theorem.

\begin{thm} Let $\n$ be a singular 2-step nilpotent Lie algebra with an open and dense set of resonant vectors $\zz_R$. If the first hit map $F_Z$ has maximal rank on a set  $\widehat{\u}_Z$ which is open and dense in  $\u_Z$ then $\GN$ has the density of closed geodesics property for every lattice $\Gamma \subseteq N$. 
\label{MaxRankImpliesDCGPThm}
\end{thm}

As an application of the theorem, consider the metric 2-step nilpotent Lie algebra $\n_G$ constructed from the path on three vertices. 

\begin{prop} Let $\n_G$ be the metric 2-step nilpotent Lie algebra associated to the graph $G$. If  $G$ is the path on three vertices, then the first hit map $F_Z$ has maximal rank for $Z$ in an open dense subset of $\zz$.
\label{ThreePathMaxRankProp}
\end{prop}

The next corollary on the density of closed geodesics follows immediately from Proposition \ref{ThreePathMaxRankProp} and Theorem \ref{MaxRankImpliesDCGPThm}.

\begin{cor} Let $\n_G$ be the metric 2-step nilpotent Lie algebra associated to the graph $G$ where $G$ is the path on three vertices. Then for any lattice $\Gamma$ in $N$, $\GN$ has the density of closed geodesics property.
\label{ThreePathMaxRankCor}
\end{cor}

\begin{proof}[Proof of Proposition \ref{ThreePathMaxRankProp}]

Let $G$ be the path on three vertices labeled as in Figure \ref{fig3path} and let $\n=\n_G$ be the associated metric 2-step nilpotent Lie algebra.

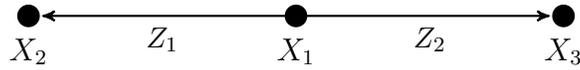
\begin{figure}[h!]
\begin{tikzpicture}[->,>=stealth',shorten >=1pt,auto,
  thick,vertex/.style={circle,draw,fill,scale=.7,font=\sffamily\Large\bfseries},node distance=2in,thick]

  \node[vertex, label=below:{$X_2$}](X2)  {};
  \node[vertex, label=below:{$X_1$}] (X1) [right of=X2]  {};
  \node[vertex, label=below:{$X_3$}] (X3)[right of=X1]  {};

\path[every node/.style={font=\sffamily\small}]
 (X1) edge node[midway,below]{$Z_1$} (X2)
       edge node[midway,below]{$Z_2$} (X3);

\end{tikzpicture}
\caption{The path on three vertices}
\label{fig3path}
\end{figure}
Then $\n= \V \oplus \zz$ where $\V$ has orthonormal basis $\{X_1, X_2, X_3\}$ and $\zz$ has orthonormal basis $\{Z_1, Z_2\}$.  Let $Z= a_1 Z_1 + a_2 Z_2$ be a nonzero element of $\zz$. In this basis for $\zz$, the map $j(Z)$ has matrix given by
$\ds j(Z) = \left( \begin{array}{ccc}
 0 & -a_1 & -a_2 \\
a_1 & 0 & 0 \\
a_2 & 0 & 0 
\end{array} \right).$
The eigenvalues of $j(Z)$ are $\{0, \pm i |Z| \}$. Observe from the eigenvalues that $\n_G$ is Heisenberg-like. This is consistent with Corollary \ref{hlike} since the path on three vertices is the star graph $K_{1,2}$. The set of resonant vectors $\zz_R \subseteq \zz$ is equal to $\zz_R = \zz - \{0\}$  and is open and dense in $\zz$. 

Decompose $\V$ into $j(Z)$-invariant subspaces $W_0$ and $W_1$ where $W_0 = \ker(j(Z)) = \spn \{\eta_1 \}$ and $W_1 = \spn \{ \eta_2, \eta_3\}$ for $\eta_1 = a_2 X_1 - a_1 X_3$, $\eta_2 = X_1$, and $\eta_3 = a_1 X_2 + a_2X_3$.  Let $\om \in \R$ be such that $e^{\om j(Z)} = Id$ on $\V$ and define the first hit map $F_Z: \u_Z \rightarrow \w_Z$ by $F_Z(\xi) = \log(\gamma_\xi (\om))$. Again we use the notation $\gamma_\xi (t)$ to denote the geodesic in $N$ with $\gamma_\xi(0) = e$, $\gamma_\xi' (t) = \xi$. 

Fix $r\neq 0$ in $\R$ and $X\in \V$ such that $\xi = X+rZ \in \u_Z$. We need to compute
\begin{eqnarray*}
F_Z(X+rZ) &=& \log(\gamma_\xi(\om))\\
 & = & X(\om) +Z(\om)
\end{eqnarray*}
where $\log(\gamma_\xi(t)) = X(t) + Z(t)$ is given by the geodesic equations shown in Proposition \ref{GeodesicEquations}. Write $X\in \V = W_0\oplus W_1$ as $X = V_1 +V_2$ where $V_1 \in W_0$ and $V_2 \in W_1$.    Then from the geodesic equations, $X(\om) = \om V_1$ and $Z(\om) = \om \widetilde{Z}_1(\om) + \widetilde{Z_2}(\om)$. Since $e^{\om j(Z)} = Id$ on $\V$ and since $j(Z)^2$ has only one nonzero eigenvalue, we see that $\widetilde{Z}_2(\om) = 0$ and 

\begin{equation} \widetilde{Z}_1(\om)  = rZ + \frac{1}{2} [V_1, 2 Id \, j(rZ)^{-1} V_2] + \frac{1}{2} [j(rZ)^{-1} V_2, V_2].
\label{3Path:Z1eqn}
\end{equation}

Write $X= V_1+V_2$ in the basis $\{\eta_1, \eta_2, \eta_3\}$ for $\V$; let $V_1 = b_1 \eta_1$ and $V_2 = b_2\eta_2 + b_3 \eta_3$ with coefficients $b_i \in \R$, $i\leq i\leq 3$. A straightforward computation gives $j(rZ)^{-1}$ on $W_1$: 

$$j(rZ)^{-1} (\eta_2) = \frac{-1}{r|Z|^2} \eta_3$$
$$j(rZ)^{-1} (\eta_3) = \frac{1}{r} \eta_2.$$

Then we have \begin{equation}j(rZ)^{-1} (V_2) = \frac{b_3}{r} \eta_2 - \frac{b_2}{r |Z|^2} \eta_3. \label{3Path:J-1}\end{equation}

We  also compute the bracket relations among $\{\eta_1, \eta_2, \eta_3\}$. 
\begin{equation} [\eta_1, \eta_2] = -a_2 Z_1 + a_1Z_2, [\eta_1, \eta_3] = 0, [\eta_2, \eta_3] = a_1 Z_1 + a_2 Z_2 = Z. \label{3Path:brackets} \end{equation}

Hence from the  (\ref{3Path:Z1eqn}), (\ref{3Path:J-1}) and (\ref{3Path:brackets}): 

\begin{eqnarray*}
\widetilde{Z}_1 (\om) &= &rZ + \frac{1}{2} \left[ b_1\eta_1, 2 \left( \frac{b_3}{r}\eta_1 - \frac{b_2}{r|Z|^2} \eta_3 \right) \right] + \frac{1}{2} \left[\frac{b_3}{r} \eta_2 - \frac{b_2}{r|Z|^2} \eta_3, b_2 \eta_2 + b_3 \eta_3 \right] \\
 & = & \left( r + \frac{b_3^2}{2r} + \frac{b_2^2}{2r|Z|^2 }\right) Z + \frac{b_1b_3}{r} ( -a_2Z_1 + a_1Z_2). 
\end{eqnarray*}

Thus, we obtain the following formula for the first hit map $F_Z$. 
\begin{eqnarray}
F_Z (X+rZ) & = & F_Z(b_1 \eta_1 + b_2\eta_2 + b_3 \eta_3+ ra_1Z_1 + ra_2Z_2) \nonumber \\
& = & \om \left\{ b_1\eta_1 + \left[ \left( r+ \frac{b_3^2}{2r} + \frac{b_2^2}{2r|Z|^2} \right) Z + \frac{b_1b_3}{r}(-a_2Z_1 + a_1Z_2 ) \right] \right\}.
\label{eq:FzFormula}
\end{eqnarray}

To show $F=F_Z: \u_Z \rightarrow \w_Z$ has maximal rank, we compute $dF_\xi (\eta)$ where $\eta \in T_\xi (\u_Z) = \spn \{ Z, \V\}$. Since $dF_\xi: \spn\{Z, \V\} \rightarrow \w_z$ maps a 4 dimensional space to a three dimensional space, it suffices to show that the dimension of the kernel of $dF_\xi$ is 1. Let $Y(s) = \xi + s \eta$ be the curve passing through $\xi = X+rZ$ with initial velocity $\eta \in \spn\{Z, \V\}$. Then $dF_\xi (\eta) = \frac{d}{ds} F(Y(s)) |_{s=0}$. 

We now compute $\frac{d}{ds} F(Y(s)) |_{s=0}$. Write $\eta$ in terms of the basis $\{\eta_1, \eta_2, \eta_3\}$ of $\V$. Let $\ds \eta =  (\Sigma_{i=1}^3 c_i \eta_i) +qZ$, $c_i, q \in \R$, $1\leq i \leq 3$. Then $$Y(s) = (b_1+sc_1) \eta_1 + (b_2 + sc_2)\eta_2 + (b_3 + sc_3) \eta_3 + (r+sq)a_1Z_1 + (r+sq)a_2Z_2.$$

From equation (\ref{eq:FzFormula})\begin{eqnarray*}
 F(Y(s)) &=& \om \left\{ (b_1+sc_1)\eta_1 + \left[ \left( (r+sq) + \frac{(b_3+sc_3)^2}{2(r+sq)} + \frac{(b_2+sc_2)^2}{2(r+sq) |Z|^2} \right) \right. \right. Z\\
 & &  \left. \left.+ \frac{(b_1+sc_1)(b_3+sc_3)}{r+sq} (-a_2Z_1 + a_1Z_2) \right]\right\}.
\end{eqnarray*}

A tedious, but straightforward computation gives 
\begin{eqnarray*}
\frac{d}{ds} F(Y(s))|_{s=0} & = & \om\left\{ c_1\eta_1 + \left[\left( q+\frac{1}{2} \left(\frac{2rb_3c_3 - c_3^2q}{r^2}\right) + \frac{1}{2|Z|^2} \left(\frac{2rb_2c_2 - c_2^2 q}{r^2}\right) \right) a_1\right. \right.\\
  & & \left.+ \left( \frac{r(b_1c_3 + b_3c_1) - b_1b_3 q}{r^2} \right) (-a_2)\right] Z_1 \\
	& & +\left[ \left( q+\frac{1}{2} \left( \frac{2rb_3c_3 - c_3^2 q}{r^2}\right) + \frac{1}{2|Z|^2} \left(\frac{2rb_2c_2 - b_2^2q}{r^2}\right) \right) a_2 \right.\\
	 & &\left. \left.+ \left(\frac{r(b_2c_3 + b_3c_1) - b_1b_3 q}{r^2}\right) a_1\right]Z_2\right\}.
\end{eqnarray*}


If we assume that $b_1\neq 0$ and $b_2 \neq 0$, another lengthy, but straightforward computation shows that \[ \ker(dF_\xi) = \spn \left\{\delta \eta_2 + \frac{b_3}{r} \eta_3 + Z\right\}\]
where \[\delta = \frac{|Z|^2r}{b_2} \left(-1-\frac{b_3^2}{2r^2} + \frac{b_2^2}{2|Z|^2r^2}\right).\] 

Therefore, $F_\xi$ has maximal rank of 3 on $\widehat{\u_Z} = \{X+rZ \, | \, r\neq 0, \, b_1\neq 0,\,  b_2 \neq 0\}$ which is open and dense in $\u_Z$. And hence by Theorem \ref{MaxRankImpliesDCGPThm}, $\n_G$ has the density of closed geodesics property when $G$ is a path on 3 vertices. 

\end{proof}

\subsection{Density of closed geodesics for star graphs} \label{stargraphs}

In this section we  prove Theorem \ref{DCGforStar4orMore}. Let $G$ be a star graph on at least four vertices. Let $\n_G$ denote the associated 2-step nilpotent Lie algebra with corresponding simply connected nilpotent Lie group $N$.  We show there exists a lattice $\Gamma \subseteq N$ such that $\Gamma \backslash N$ has the density of closed geodesics property.

Let $G$ denote the star graph $K_{1,n}$ with vertices  $X_1, \ldots X_{n+1}$ where $n \geq 3$  and edges $Z_1, \ldots, Z_n$ as in Example \ref{singex}.
Let $\n_G$ denote the associated metric 2-step nilpotent Lie algebra with orthonormal basis $\beta = \{X_1, \ldots X_{n+1},Z_1, \ldots Z_n\}$ and let $N$ denote the simply connected nilpotent Lie group with Lie algebra $\n_G$. Then in $\n_G$, we have bracket relations $Z_{i-1} = [X_1, X_i]$ for all $i = 2, \ldots n+1$. 
 As before, we write $\n_G$ as $\n_G = \V \oplus \zz$ where $\V = \spn\{X_1, \ldots, X_{n+1}\}$, $\zz= \spn\{Z_1, \ldots, Z_n\}$. If  $Z = \sum_{i=1}^{n} a_i Z_i$, for $a_i \in \R$, $1\leq i \leq n$, then the matrix of $j(Z)$  with respect to a basis $\{X_1, \ldots, X_{n+1}\}$ is given by

$$j(Z) = \left(\begin{array}{cccccc} 0 & -a_1 & -a_2 &   \cdots & -a_{n-1} & -a_{n} \\  a_1 & 0 &  0 & \cdots & 0 & 0 \\ a_2 & 0 &   0 & \cdots & 0 & 0 \\ \vdots & \vdots &   \vdots & \vdots & \vdots & 0 \\ a_{n-1} & 0  & 0 &  \cdots  & 0  &  0 \\ a_{n}  & 0 & 0 &\cdots  & 0 &  0 
\end{array}\right).$$
The eigenvalues of $j(Z)$ are $0$ with multiplicity $n-1$ and $\pm i |Z|$.   Since every non-zero $Z \in \zz$ is in resonance,  for each $Z$ there exists a positive real number $\omega$  such that $e^{\omega j(Z)} = Id$ on $\V$ (see Proposition \ref{EResonanceProp}).  In fact, in this case it can be seen that  $\omega  = \frac{2 \pi }{|Z|} $ satisfies 
$e^{\omega j(Z)} = Id$  on $\V$ and further, $e^{m\omega j(Z)} = Id$ for all $m \in \mathbb N$ for each non-zero $Z\in \zz$.

Let $0 \neq Z =  \sum_{i=1}^{n} a_i Z_i \in \zz$. We write $\V = W_0 \oplus W_1$ where  $W_0$ denotes the $\ker j(Z)$ and $W_1$ is the orthogonal complement of $W_0$.  Then $W_0 = \Span \{\eta_1, \ldots, \eta_{n-1} \}$ where $\eta_i = -a_{i+1}X_2 + a_1 X_{i+2}$ for $i = 1, \ldots, n-1$ and   $W_1 = \Span \{  \eta_n, \eta_{n+1} \}$ where $\eta_n = X_1$ and $\eta_{n+1} = \sum_{i=1}^{n} a_{i}X_{i+1}$.  We also note that $j(Z) (\eta_n) = \eta_{n+1}$ and $j(Z)(\eta_{n+1}) = - |Z|^2 \eta_n$.  Hence the matrix of the inverse of $j(rZ)$ for $r\neq 0$ on $W_1$ with respect to basis $\{\eta_n, \eta_{n+1}\}$ is given by $j(rZ)^{-1}|_{W_1} = \left( \begin{array}{cc} 0 & \frac{1}{r} \\ \frac{-1}{r|Z|^2} & 0 \end{array}\right) $. The Lie brackets are given by \[ [\eta_i, \eta_k] = 0,\,\,
 [\eta_i, \eta_{n+1}] = 0,\,\,
[\eta_i, \eta_n] = a_{i+1} Z_1 - a_1 Z_{i+1}  \,\, \mbox{ for all } i,k = 1, \ldots, n-1, \]
\[ [ \eta_n, \eta_{n+1}] = Z.\]


We will follow the same notation from subsection \ref{subsec:FirstHitApproach}.  
For a fixed  $0 \neq Z  \in \zz$, we define    $\u_Z = \{X+rZ \,|\, r\neq 0, X\mbox{ has nonzero component in  }\ker j(Z) \}$ and $\w_Z = \zz \oplus \ker j(Z)$.  For $\xi \in \n_G$, let  $\gamma_\xi(t)$ denote the geodesic in $N$ with $\gamma_\xi(t) = e$ and $\gamma'_\xi(0) = \xi$.  Consider the first hit map $F_Z: \u_Z \rightarrow \w_Z$ given by $F_Z(\xi) = \log(\gamma_\xi(\omega))$, where $\omega \in \mathbb R$ is such that $e^{\omega j(Z)} = Id$  on $\V$.  Now let $\xi \in \u_Z$ and  $\xi = X + r Z$ for some $r \neq 0$. We write $X$ in terms of the new basis for $\V$ as $X = \sum_{i=1}^{n+1} b_i\eta_i$ with $b_i \in \R$ for $1\leq i \leq n+1$. Write $Z = \sum_{i=1}^n a_i Z_i$, with $a_i\in\R$, $1\leq i \leq n$. Using the geodesic equations given in Proposition \ref{geoeqns}, 
we compute $F_Z = \log(\gamma_\xi (\omega))$ to obtain the following expression.

\begin{equation}\label{GeodesicStarOn4}
F_Z (X+rZ) = \omega \left \{\sum_{i=1}^{n-1} b_i \eta_i + \left(r + \frac{b_{n+1}^2}{2r} + \frac{b_n^2}{2 r^2 |Z|^2}\right) Z  +
 \sum_{i=1}^{n-1} \frac{b_i b_{n+1}}{r}(a_{i+1} Z_1 - a_1 Z_{i+1}) \right\}.
\end{equation}

The first hit map $F_Z$ maps $\u_Z$, a $n+2$ dimensional space, to $\w_Z$, a $2n-1$ dimensional space. In the case $n=3$, $F_Z$ does not have full rank, and clearly for $n>3$ the rank of $F_Z$ is at most $n+2$ which is less than $2n-1$. So the first hit map approach, previously used to show the density of closed geodesics property holds in $\GN$ for any lattice $\Gamma$ fails in this case. 

Let $\Gamma$ denote the lattice in $N$ given by $\Gamma = \exp(\Lambda)$ for $\Lambda = \spn_{2\pi \Z} \{\beta\}$ where $\beta$ is the orthonormal basis of $\n_G$ given by the vertices and edges of the graph $G$.  Let $Z$ be any element in  $\spn_\Q\{Z_1, \ldots, Z_n\}$ such that $|Z| \in \mathbb Q$. Choose any $\xi \in \u_Z$ so that $\xi = X + rZ$ with  $r \in \Q$ and $X = \Sigma_{i=1}^{n+1} b_i \eta_i$  where $b_i\in\Q$ for $1\leq i \leq n_1$. Expand the equation  (\ref{GeodesicStarOn4}) in terms of the basis $\beta$ to see that for all $m \in \N$, the $m^{th}$ hit map is given by  $F^m_ Z(\xi) =  \log(\gamma_\xi(m\om ) ) = \frac{2 \pi m}{|Z|} Y$ where  $Y \in \spn_\Q \{\beta \}$.  Hence $\log(\gamma_\xi(m\om) ) \in \Lambda$ for sufficiently large  $m\in \N$.

Note that the set of $Z \in \zz \cap \spn_\Q\{Z_1, \ldots, Z_n\}$ such that $|Z| \in \mathbb Q$  is dense in $\zz$ which we justify by the following lemma.

\begin{lemma}
The set $\{ \textbf{v}  \in \mathbb Q^n \,\,|\,\, | \textbf{v} | \in \mathbb Q \}$  is dense in $\mathbb R^n$.

\end{lemma}

\begin{proof} Let $\textbf{u} \in \mathbb R^n$, $\textbf{u} \neq 0$ and $\epsilon >0$.  Let  $S^{n-1}$ denote the unit sphere in $\mathbb R^n$. Note that $\mathbb Q^n \cap S^{n-1}$ is dense in $S^{n-1}$, which can be seen through stereographic projection. Thus, there exists $\textbf{w} \in \mathbb Q^n \cap S^{n-1}$ such that $|\textbf{w} - \frac{\textbf{u}}{|\textbf{u}|}\,| < \frac{\epsilon}{2|\textbf{u}|}$. Also there exists $s \in \mathbb Q$ such that $|s - |\textbf{u}| \,| < \frac{\epsilon}{2}$.  We then have $s\textbf{w} \in \mathbb Q^n$ and 
\begin{eqnarray*}
|s\textbf{w} - \textbf{u}| &\leq& |s\textbf{w} - |\textbf{u}| \textbf{w} \, | + |\, |\textbf{u}| \textbf{w} - \textbf{u}| \\
&\leq&  |s - |\textbf{u}| | + |\textbf{u}| |\, \textbf{w} -  \frac{\textbf{u}}{|\textbf{u}|}|  \,\, \,\,\,\,\text{ since } |\textbf{w} | = 1\\
&<& \epsilon.
\end{eqnarray*} 
\end{proof}

Therefore, there is a dense set of vectors $\xi \in \n_G$ such that the geodesic $\gamma_\xi(t)$ in $N$ is translated by a lattice element in the lattice $\Gamma = \exp(\spn_{2\pi\Z} \{\beta \})$. This proves Theorem \ref{DCGforStar4orMore}, stated below.

\begin{thm} \label{DCGforStar4orMore}
Let $\n_G$ denote the metric 2-step nilpotent Lie algebra determined to the star graph $G = K_{1,n}$ where $n \geq 3$ and let $N$ denote the associated simply connected nilpotent Lie group. Let $\beta$ denote the orthonormal basis determined by the graph, and let $\Gamma$ be the lattice in $N$ given by $\Gamma = \exp(\spn_{2\pi \Z} \{\beta \})$. Then the quotient $\Gamma \backslash N$ has the density of closed geodesics property.

\end{thm}

\subsection{Density of closed geodesics for $K_3$} \label{K3graphSubsection}

We follow the procedure of subsection \ref{stargraphs} in the case where $G=K_3$. Let $\n_G$ be the metric 2-step nilpotent Lie algebra constructed from the directed graph $K_3$ with labels shown in Figure \ref{k3fig} with basis $\beta= \{X_1, X_2, X_3, Z_1, Z_2, Z_3\}$. 

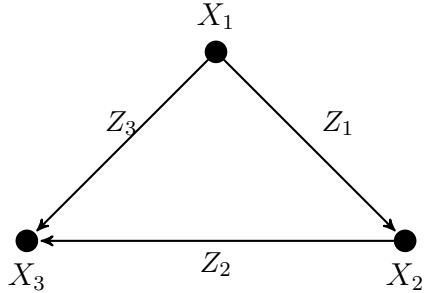
\begin{figure}[h!]
\begin{tikzpicture}[->,>=stealth',shorten >=1pt,auto,
  thick,vertex/.style={circle,draw,fill,scale=.7,font=\sffamily\Large\bfseries},node distance=2in,thick]

  \node[vertex, label=above:{$X_1$}](X1)  {};
  \node[vertex, label=below:{$X_2$}] (X2) [below right of=X1]  {};
  \node[vertex, label=below:{$X_3$}] (X3)[below left of=X1]  {};

\path[every node/.style={font=\sffamily\small}]
 (X1) edge node {$Z_1$} (X2)
       edge node [midway, above] {$Z_3$} (X3)
(X2) edge node {$Z_2$} (X3);
\end{tikzpicture}

\caption{The graph of $K_3$}
\label{k3fig}
\end{figure}

\begin{thm} Let $\n_G$ be the metric 2-step nilpotent Lie algebra determined by the complete graph $K_3$ and let $N$ be the associated simply connected 2-step nilpotent Lie group. Let $\beta$ denote the orthonormal basis determined by the graph, and let $\Gamma$ be the lattice in $N$ given by $\Gamma = \exp \left(\spn_{2\pi \Z} \{ \beta\}\right)$. Then the quotient $\GN$ has the density of closed geodesics property. 
\label{thm:K3DCGP}
\end{thm}

\begin{proof}

  Let $Z= a_1Z_1+a_2Z_2 + a_3Z_3 \in \zz$. Then $j(Z) = \left( \begin{array}{ccc} 
0 &-a_1 & -a_3\\
a_1 & 0 & -a_2 \\
a_3 & a_2 & 0 
\end{array} \right)$.

Let $B=\om j(Z)$. The Rodrigues formula 
gives \[ e^B = I_3 + \frac{\sin(\theta)}{\theta} B + \frac{1-\cos(\theta)}{\theta^2} B^2,\] where $\ds \theta = \om|Z|$. It follows directly that $e^{\om j(Z)} = Id$ when $\ds \om = \frac{2\pi n}{|Z|}$ for $n\in \N$. 

The map $j(Z)$ has eigenvalues $\{0, \pm i|Z|\}$. Note that $\n$ is Heisenberg-like by Corollary \ref{hlike} and that every nonzero element $Z\in\zz$ is in resonance. We decompose $\V = \zz^{\perp}$ into $j(Z)$  invariant subspaces $W_0 = \ker j(Z) = \spn\{ \eta_1\}$ and $W_1 = W_0^{\perp} = \spn\{ \eta_2, \eta_3 \}$ where  
\[\eta_1 = a_2 X_1 - a_3X_2+a_1X_3, \,   \eta_2 = -a_1 X_1+a_2X_3, \,  \eta_3 = a_2a_3X_1 + (a_1^2+a_2^2)X_2 + a_1a_3X_3\]

A straightforward computation gives $j(Z) \eta_2 = -\eta_3$ and $j(Z) \eta_3 = |Z|^2 \eta_2$. If $r\in \R$ is nonzero, then with respect to the basis $\{\eta_1, \eta_2\}$ for $W_1$,  we obtain the matrix $j(rZ)^{-1} |_{W_1} = \left( \begin{array}{cc} 0 & \frac{1}{r} \\ \frac{-1}{r|Z|^2} & 0 \end{array}\right) $ and the brackets among $\{\eta_1, \eta_2,\eta_3\}$ are given:
\[ [\eta_1, \eta_2] = -a_1a_3Z_1 - a_2a_3Z_2+(a_1^2+a_2^2)Z_3 \]
\[ [\eta_1,\eta_3] = |Z|^2(a_2Z_1 - a_1Z_2)\]
\[ [\eta_2, \eta_3] = -a_1(a_1^2+a_2^2)Z_1 - (a_1^2a_3 +a_2(a_1^2+a_2^2))Z_2 -a_2^2a_3Z_3.\]

Following the strategy used in Subsection \ref{stargraphs}, we compute the first hit map $F_Z$ for $Z=a_1Z_1+a_2Z_2+a_3Z_3$ where $a_1, a_2, a_3 \in \R$. 

Since every nonzero $Z\in \zz$ is in resonance, there is a positive real number $\om$ with $e^{\om j(Z)} = Id$ on $\V$. Let $\u_Z = \{X+rZ \,|\, r>0, X\mbox{ has nonzero component in  }\ker j(Z) \}$ and $\w_Z = \zz \oplus \ker j(Z)$ and let $\xi \in \u_Z$. Then $\xi= X+rZ$ where $X= \sum_{i=1}^3 b_i \eta_i$ for $b_1, b_2, b_3 \in \R$.    The first hit map $F_Z: \u_Z \rightarrow \w_Z$ is given by $F_Z(\xi) = \log(\gamma_\xi (\om)) = X(\om)+Z(\om)$ where $X(\om)$ and $Z(\om)$ are found after a straightforward computation using the geodesic equations from Proposition \ref{GeodesicEquations}: 
\begin{eqnarray}
X(\om) &=& \om b_1 \eta_1 \nonumber\\
Z(\om) &=& \om \widetilde{Z_1} + \widetilde{Z_2}(\om) \mbox{ where }\\
 & & \widetilde{Z}_1(\om) = \left( ra_1-\frac{b_1b_2}{r}a_2 - \frac{b_1b_2}{r} a_1a_2 + \frac{1}{2r}(b_2^2+b_2b_3)(a_1)(a_1^2+a_2^2)\right) Z_1\nonumber\\
 & & + \left(ra_2+\frac{b_1b_2}{r}a_2a_3 + \frac{a}{2r} (b_2^2-b_2b_3)(a_1^2a_3 - a_2(a_1^2+a_2^2)) \right) Z_2 \nonumber\\
 & & + \left( ra_3 + \frac{b_1b_2}{r}(a_1^2+a_2^2) + \frac{1}{2r} ( b_2^2+b_2b_3)(a_2^2a_3) \right) Z_3\nonumber\\
 & & \widetilde{Z}_2(\om) = 0. \nonumber
\end{eqnarray}

The map $F_Z: \u_Z \rightarrow \w_Z$ is a map between 4 dimensional vector spaces. However, this map fails to have maximal rank. As in subsection \ref{stargraphs}, we can not prove density of closed geodesics in $\GN$ for any lattice $\Gamma$, and instead will establish the density of closed geodesics in $\GN$ for a particular lattice. 

Let $\Gamma$ be the lattice in $N$ given by $\exp (\Lambda)$ where $\Lambda$ is the vector lattice in $\n$ given by $\Lambda =  \spn_{2\pi\Z} \{\beta\}$. Let $A=\{Z\in\zz \, | \, |Z|\in \Q\}$ and choose $Z\in A \cap \spn_\Q \{Z_1, Z_2, Z_3\}$. This ensures that the coefficients $a_1, a_2, a_3 \in \Q$ and that $|Z| \in \Q$. Choose $\xi \in \u_Z$ so that $\xi = X+rZ$ with $r\in \Q$ and  $X\in \V$ so $X = \sum_{i=1}^3 b_i \eta_i$ where $b_i \in \Q$ for $1\leq i \leq 3$. Since the coefficients of $F_Z (\xi)$ are polynomial in $a_i, b_j$, we can choose $m\in \N$ sufficiently large so that $\ds F_Z(\xi) = \log \left( \gamma_{\xi} \left(\frac{2\pi m}{|Z|}\right) \right) \in \Lambda$.  

The set of $\xi \in \n$ with $\xi = X+rZ$ with $X \in \spn_\Q\{X_1, X_2, X_3\}$, $Z\in\spn_\Q\{Z_1, Z_2, Z_3\}$, $|Z|\in\Q$, $r\in \Q$ are dense in $\n$. Therefore, for a dense set of $\xi \in \n$, the geodesic $\gamma_\xi(t)$ in $N$ hits a lattice element and is translated by that lattice element by amount $\om$. Thus there is a dense set of closed geodesics in $\GN$.  
\end{proof}

\section{Proof of Proposition~{\ref{k4res}}}
\label{pfk4res}

\begin{proof}  First consider the Lie algebra associated to the star graph $K_{1,3}$.  For any $Z\in \zz$, $j(Z)$ is singular and the eigenvalues are $\{0,\pm i|Z|\}$.  Thus all  nonzero $Z\in \zz$ are in resonance.

For any other connected graph on four vertices we mimic the proof of Lee-Park in \cite{LP} to show that we have a dense set of resonant vectors.

We start by considering the 10-dimensional 2-step nilpotent Lie algebra associated to $K_4$, as described in subsection \ref{k4graphs}.  For any $Z\in\zz$, we write $Z=a_1Z_1+a_2Z_2+a_3Z_3+a_4Z_4+a_5Z_5+a_6Z_6$, $a_i \in \R$ for $1 \leq i\leq 6$, and find the $j(Z)$ matrix to be $\left(\begin{array}{cccc} 0 & -a_1 & -a_2 & -a_3 \\ a_1 & 0 & -a_4 & -a_5 \\ a_2 & a_4 & 0 & -a_6 \\ a_3 & a_5 & a_6 & 0 \end{array}\right)$.  The eigenvalues of $j(Z)$ are $\displaystyle\left\{\pm\frac{i}{2}\sqrt{\alpha + \sqrt{\beta}}, \
 \pm\frac{i}{2}\sqrt{\alpha - \sqrt{\beta}}\right\}$, where
\begin{eqnarray}\alpha&=&|Z|^2=a_1^2+a_2^2+a_3^2+a_4^2+a_5^2+a_6^2  \label{alp}\\
\beta&=&|Z|^4-4(a_1a_6+a_3a_4-a_2a_5)^2.  \label{bet}
\end{eqnarray}
Letting $a_0=a_1a_6+a_3a_4-a_2a_5$, we write $\beta =|Z|^4-4a_0^2$. See \cite{P}.

To each element $Z\in\zz$, written as above,  we associate the vector $\mb{v}\in\mathbb{R}^6$, defined as $\mb{v}=(a_1,a_2,a_3,a_4,a_5,a_6)$.  Through this identification, we have established a one-to-one mapping between $\zz$ and $\mathbb{R}^6$.  
To consider the ratio of eigenvalues of $j(Z)$ for a given nonzero $Z\in\zz$, we define the map $g:\mathbb{R}^6\rightarrow \mathbb{R}$ as $\displaystyle g(\mb{v})=\frac{\alpha+\sqrt{\beta}}{\alpha-\sqrt{\beta}}$, where $\alpha$ and $\beta$ are defined in (\ref{alp}) and (\ref{bet}) in terms of the $a_i$.  
We are only concerned with the ratio of nonzero eigenvalues, as defined by resonance.  Note that if $\alpha=\sqrt{\beta}$, then the eigenvalues of $j(Z)$ are $\left\{0,\pm i\sqrt{\frac{\alpha}{2}}\right\}$ and it is clear that $Z$ is in resonance without the use of the $g$ map. We consider the gradient of the map $g$  where $\sqrt{g}$ is the ratio of any two eigenvalues.    We show that the gradient $\nabla g(\mb{v})$ has rank 1 on a dense open subset of $\mathbb{R}^6$.  Define this set to be $V=\{\mb{v}\in \mathbb{R}^6| \nabla g(\mb{v})\neq \mb{0}\}$.
 Therefore, as in \cite{LP}, the inverse image of the rationals in $\mathbb{R}$ under $\sqrt{g}$ is dense in $\mathbb{R}^6$ and therefore in $\zz$.   Thus $j(Z)$ will be in resonance for a dense set of $Z$ in $\zz$. 

In calculating $\nabla g(\mb{v})$ we find $\ds \frac{\partial g}{\partial a_i}$ for $i=1,\dots, 6$.  In terms of $\alpha$ and $\beta$, 
\begin{eqnarray} \frac{\partial g}{\partial a_i}&=&\frac{1}{\sqrt{\beta}(\alpha-\sqrt{\beta})^2}\left(\alpha \frac{\partial \beta}{\partial a_i} -2\beta \frac{\partial \alpha}{\partial a_i} \right) \label{partgi} \end{eqnarray}
If $\nabla g(\mb{v}) = \mathbf{0}$ for some $\mb{v}\in\mathbb{R}^6$, then $\ds \partg=0$ for all $i=1,\dots,6$ and from (\ref{partgi}) $\ds \alpha \frac{\partial \beta}{\partial a_i} -2\beta \frac{\partial \alpha}{\partial a_i} =0$. We show that for some $i$ the set of $\mb{v}$ such that $\ds \frac{\partial g}{\partial{a_i}}(\mb{v})\neq 0$ is a dense open subset of $\mathbb{R}^6$. We will do this by calculating $\ds \frac{\partial g}{\partial{a_1}}(\mb{v})$.

In the case of $K_4$, $\ds \frac{\partial \beta}{\partial a_1}=4a_1|Z|^2-8a_6a_0$ and $\ds  \frac{\partial \alpha}{\partial a_1}=2a_1$.  
Substituting into (\ref{partgi}),
\begin{eqnarray*} \frac{\partial g}{\partial a_1}&=&\frac{1}{\sqrt{\beta}(\alpha-\sqrt{\beta})^2}\left(\alpha \frac{\partial \beta}{\partial a_1} -2\beta \frac{\partial \alpha}{\partial a_1} \right)\\
&=&\frac{1}{\sqrt{\beta}(\alpha-\sqrt{\beta})^2}\left(|Z|^2 (4a_1|Z|^2-8a_6a_0) -2(|Z|^4-4a_0^2)(2a_1) \right)\\
&=&\frac{1}{\sqrt{\beta}(\alpha-\sqrt{\beta})^2}a_0\left(-8a_6|Z|^2 +16a_0a_1 \right)
\end{eqnarray*}
If $\ds \frac{\partial g}{\partial a_1}=0$, then  $a_0\left(-a_6|Z|^2 +2a_0a_1 \right)=0$.  Note that $a_0\left(-a_6|Z|^2 +2a_0a_1 \right)=0$ is a fifth degree polynomial since $a_0=a_1a_6+a_3a_4-a_2a_5$ and $|Z|^2=a_1^2+a_2^2+a_3^2+a_4^2+a_5^2+a_6^2$, showing that the set of $\mb{v}$ on which $\ds \frac{\partial g}{\partial{a_1}}(\mb{v})=0$ is a closed algebraic variety since it is the zero set of a polynomial in the variables  $a_1,a_2,\dots,a_6$. 

The same argument can be extended to show the existence of a dense set of resonant vectors associated to each of the subgraphs of $K_4$ by considering the map $g:\mathbb{R}^n\rightarrow \mathbb{R}$ for the appropriate $n$, as indicated below.  In each case we will show that the appropriately defined set $V=\{\mb{v}\in \mathbb{R}^n| \nabla g(\mb{v})\neq \mb{0}\}$ is precisely a closed algebraic variety, as in the $K_4$ case.  
In each case where $G$ is a proper subgraph of $K_4$, $j(Z)$ is defined similarly, with $a_i=0$ for appropriate $i$ as defined by the subgraph.   For a generic $Z\in\zz$, the eigenvalues of $j(Z)$ are always $\displaystyle\left\{\pm\frac{i}{2}\sqrt{\alpha + \sqrt{\beta}},\  \pm\frac{i}{2}\sqrt{\alpha -\sqrt{\beta}}\right\}$, where $\alpha$ and $\beta$ are given in (\ref{alp}) and (\ref{bet}), again, with $a_k=0$ for all removed edges $Z_k$.  In each case we can define a corresponding $a_0$ and find an open and dense set for which  $\ds \frac{\partial g}{\partial a_1}\neq 0$ by showing that  $\ds \frac{\partial g}{\partial a_1}$ is a polynomial in the remaining $a_i$.  We summarize these details below for each case. 

\noindent \textbf{Case} $G_1$: $Z=a_1Z_1+a_2Z_2+a_3Z_3+a_4Z_4+a_5Z_5$.  Define $g:\mathbb{R}^5\rightarrow \mathbb{R}$ as above where
$\ds \alpha=|Z|^2=a_1^2+a_2^2+a_3^2+a_4^2+a_5^2$ and $\ds \beta=|Z|^4-4(a_3a_4-a_2a_5)^2$.  Here $a_0=a_3a_4-a_2a_5$. 
Then  $\ds \frac{\partial \beta}{\partial a_1}=4a_1|Z|^2$ and $\ds  \frac{\partial \alpha}{\partial a_1}=2a_1$.  
Plugging into (\ref{partgi}),
\begin{eqnarray*} \frac{\partial g}{\partial a_1}&=&\frac{1}{\sqrt{\beta}(\alpha-\sqrt{\beta})^2}\left(\alpha \frac{\partial \beta}{\partial a_1} -2\beta \frac{\partial \alpha}{\partial a_1} \right)\\
&=&\frac{1}{\sqrt{\beta}(\alpha-\sqrt{\beta})^2}\left(|Z|^2 (4a_1|Z|^2) -2(|Z|^4-4a_0^2)(2a_1) \right)\\
&=&\frac{1}{\sqrt{\beta}(\alpha-\sqrt{\beta})^2}\left(16a_0^2a_1 \right)\\
\end{eqnarray*}
Clearly $\nabla g(\mb{v})=\mb{0}$ only if $a_0^2a_1=0$ where $a_0^2a_1$ is a fifth degree polynomial.  

\noindent \textbf{Case} $G_2$:  $Z=a_1Z_1+a_2Z_2+a_3Z_3+a_4Z_4$.  Define $g:\mathbb{R}^4\rightarrow \mathbb{R}$ as above where
$\ds \alpha=|Z|^2=a_1^2+a_2^2+a_3^2+a_4^2$ and $\ds \beta=|Z|^4-4(a_3a_4)^2$.  Here $a_0=a_3a_4$.  Then $\ds \frac{\partial \beta}{\partial a_1}=4a_1|Z|^2$ and $\ds  \frac{\partial \alpha}{\partial a_1}=2a_1$.  Here, the calculation of $\ds \frac{\partial g}{\partial a_1}$ is the same as in the $G_1$ case, with the result 
\begin{eqnarray*} \frac{\partial g}{\partial a_1}&=&\frac{1}{\sqrt{\beta}(\alpha-\sqrt{\beta})^2}\left(16a_0^2a_1 \right)
\end{eqnarray*}
 Clearly $\nabla g(\mb{v})=\mb{0}$ only if $a_0^2a_1=0$, where $a_0^2a_1$ again is a fifth degree polynomial.  


\noindent \textbf{Case}  $C_4$:   $Z=a_1Z_1+a_3Z_3+a_4Z_4+a_6Z_6$.  Define $g:\mathbb{R}^4\rightarrow \mathbb{R}$ as above where  
$\ds \alpha=|Z|^2=a_1^2+a_3^2+a_4^2+a_6^2$ and $\ds \beta=|Z|^4-4(a_1a_6+ a_3a_4)^2$.  Here $a_0=a_1a_6+a_3a_4$.  Then $\ds \frac{\partial \beta}{\partial a_1}=4a_1|Z|^2-8a_6a_0$ and $\ds  \frac{\partial \alpha}{\partial a_1}=2a_1$.  Here the calculation of $\ds \frac{\partial g}{\partial a_1}$ is the same as in the $K_4$ case.  We find
\begin{eqnarray*} \frac{\partial g}{\partial a_1}&=&
\frac{1}{\sqrt{\beta}(\alpha-\sqrt{\beta})^2}a_0\left(-8a_6|Z|^2 +16a_0^2a_1 \right)\\
\end{eqnarray*}
 Thus if $\ds \frac{\partial g}{\partial a_1}=0$, then $a_0\left(-a_6|Z|^2 +2a_0a_1 \right)=0$ where $a_0\left(-a_6|Z|^2 +2a_0a_1 \right)$ is a  fifth degree polynomial.

\noindent \textbf{Case} $P_4$:  $Z=a_1Z_1+a_3Z_3+a_4Z_4$.  Define $g:\mathbb{R}^3\rightarrow \mathbb{R}$ as above where
$\ds \alpha=|Z|^2=a_1^2+a_3^2+a_4^2$ and $\ds \beta=|Z|^4-4(a_3a_4)^2$.  Here $a_0=a_3a_4$, as in the $G_2$ case.  Then the calculation mirrors that one where $\ds \frac{\partial \beta}{\partial a_1}=4a_1|Z|^2$ and $\ds  \frac{\partial \alpha}{\partial a_1}=2a_1$ and
\begin{eqnarray*} \frac{\partial g}{\partial a_1}&=&\frac{1}{\sqrt{\beta}(\alpha-\sqrt{\beta})^2}\left(16a_0^2a_1 \right)
\end{eqnarray*}
 Clearly $\nabla g(\mb{v})=\mb{0}$ only if $a_0^2a_1=0$, where $a_0^2a_1$ again is a fifth degree polynomial. 


\end{proof}

\noindent{\bf Acknowledgements. }   Meera  Mainkar was supported by the Central Michigan University ORSP Early Career Investigator (ECI) grant \#C61940.

\end{document}